\documentclass[12pt]{article}
\usepackage{framed}
\RequirePackage{kpfonts}
\RequirePackage[sf,bf,small,raggedright]{titlesec}
\usepackage{textcomp}
\usepackage{amssymb}
\usepackage{bbm}
\usepackage{fancyhdr}
\usepackage{amsmath}
\usepackage{amsbsy,amsthm}
\usepackage{amscd}
\usepackage{latexsym}
\usepackage{graphicx}   
\usepackage{pdfsync}
\usepackage{blkarray}
\usepackage{multirow}
\usepackage{hyperref}
\usepackage{tcolorbox}

\newcommand{\R}{\mathbb{R}}
\newcommand{\Rd}{\R^d}
\newcommand{\inr}[1]{\left\langle #1 \right\rangle}

\newcommand{\E}{\mathbb{E}}

\newcommand{\eps}{\varepsilon}

\newtheorem{lemma}{Lemma}
\newtheorem{theorem}{Theorem}
\newtheorem{corollary}{Corollary}
\newtheorem{proposition}{Proposition}

\newtheorem{remark}{Remark}
\newtheorem{definition}{Definition}

\numberwithin{equation}{section}

\def \proof {\noindent {\bf Proof.}\ \ }
\def \remark {\noindent {\bf Remark.}\ \ }

\def \endproof
{{\mbox{}\nolinebreak\hfill\rule{2mm}{2mm}\par\medbreak}}
\def\IND{\mathbbm{1}}

\newcommand{\ol}{\overline}
\newcommand{\wt}{\widetilde}
\newcommand{\wh}{\widehat}

\newcommand{\X}{\mathcal{X}}

\newcommand{\EXP}{\mathbb{E}}
\newcommand{\PROB}{\mathbb{P}}

\newcommand{\var}{\mathrm{Var}}
\newcommand{\F}{{\mathcal F}}

\newcommand{\Tr}{\mathrm{Tr}}

\newcommand{\defeq}{\stackrel{\mathrm{def.}}{=}}

\begin{document}

\title{Multivariate mean estimation with direction-dependent accuracy
\thanks{
G\'abor Lugosi was supported by
the Spanish Ministry of Economy and Competitiveness,
Grant MTM2015-67304-P and FEDER, EU,
and by ``Google Focused Award Algorithms and Learning for AI''.
}
\thanks{Shahar Mendelson would like to thank Jungo Connectivity for its support.}
}
\author{
G\'abor Lugosi\thanks{Department of Economics and Business, Pompeu
  Fabra University, Barcelona, Spain, gabor.lugosi@upf.edu}
\thanks{ICREA, Pg. Lluís Companys 23, 08010 Barcelona, Spain}
\thanks{Barcelona Graduate School of Economics}
\and
Shahar Mendelson
\thanks{Mathematical Sciences Institute, The Australian National University,  shahar.mendelson@anu.edu.au}
}

\maketitle

\begin{abstract}
We consider the problem of estimating the mean of a random vector
based on $N$ independent, identically distributed observations.
We prove the existence of an estimator that has a near-optimal error in
all directions in which the variance of the one dimensional marginal of the random
vector is not too small: with probability $1-\delta$, the procedure returns
$\wh{\mu}_N$ which satisfies that for every direction $u \in S^{d-1}$,
\[
\inr{\wh{\mu}_N - \mu, u}\le  \frac{C}{\sqrt{N}} \left( \sigma(u)\sqrt{\log(1/\delta)} + \left(\E\|X-\EXP X\|_2^2\right)^{1/2} \right)~,
\]
where $\sigma^2(u) = \var(\inr{X,u})$ and $C$ is a constant.
To achieve this, we require only slightly
more than the existence of the covariance matrix, in the form of a
certain moment-equivalence assumption.

The proof relies on novel bounds for the ratio of empirical and
true probabilities that hold uniformly over certain classes of
random variables.
\end{abstract}

\newpage


\section{Introduction}

The problem of estimating the mean of a high-dimensional random vector with a possibly heavy-tailed distribution is a classical question that has been studied extensively over the years. Recently it has received the attention of mathematical statisticians and theoretical computer scientists. We refer to Lugosi and
Mendelson \cite{LuMe19} for a recent survey and to
Bahmani \cite{Bah20},
Dalalyan and Minasyan \cite{DaMi20},
Diakonikolas, Kane, Pensia \cite{DiKaPe20},
Minsker and Mathieu \cite{MiMa19},
Minsker and Ndaoud \cite{MiNd20}  for a sample of even more recent references.

To formulate the problem, let $X_1,\ldots,X_N$ be independent, identically distributed random
vectors taking values in $\Rd$ with mean $\EXP X = \mu \in \Rd$ (where $X$ is distributed as $X_1$.)
One is interested in constructing
a measurable function $\wh{\mu}_N: (\Rd)^N \to \Rd$ such that $\wh{\mu}_N = \wh{\mu}_N(X_1,\ldots,X_N)$ is close, in some sense,
to the mean $\mu$. A possible meaningful goal is to find an estimator such that, given a confidence parameter $\delta$, satisfies that with probability at least $1-\delta$,
the Euclidean distance $\|\wh{\mu}_N - \mu\|$ is as small as possible.

Constructing an optimal $\wh{\mu}_N$ is not that obvious even when $d=1$. Without any further assumptions on the distribution of the $X_i$ it is impossible to give meaningful performance guarantees even in that case, let alone for a random vector in $\Rd$. However, under minimal conditions, it is possible to construct estimators with remarkably strong properties. The most
common assumption is that the $X_i$ have finite second moment. When $d=1$ one can find $\wh{\mu}_N(\delta)$ such that, with probability at least $1-\delta$,
\begin{equation} \label{eq:sub-gauss-d-1}
|\wh{\mu}_N(\delta) - \mu| \leq c\sigma\sqrt{\frac{\log(1/\delta)}{N}}~,
\end{equation}
where $\sigma^2$ is the variance of $X$ and $c$ is a universal constant (see, e.g., \cite{LuMe19}).

The meaning of \eqref{eq:sub-gauss-d-1} is that even if $X$ is heavy-tailed, $\wh{\mu}_N$ performs as if $X$ were a Gaussian random variable and the estimator were the empirical mean $N^{-1}\sum_{i=1}^N X_i$. Obviously, the empirical mean does not exhibit such a behavior unless $X$ is actually Gaussian (or sub-Gaussian), which indicates that $\wh{\mu}_N$ has to be rather carefully chosen when $X$ is an arbitrary random variable.

The problem for $d>1$ is significantly more complex, though it does have a satisfying answer. Let $X \in \Rd$ and assume (as we do throughout this article) that the covariance matrix of the $X_i$, denoted by $\Sigma = \EXP (X-\mu)(X-\mu)^T$, exists. Quite remarkably, there are
mean estimators that, under this minimal condition, achieve  ``sub-Gaussian'' performance. In particular,
for any $\delta \in (0,1)$, there exists an estimator $\wh{\mu}_N=\wh{\mu}_N(\delta)$ such that, with probability at least $1-\delta$,
\begin{equation}
\label{eq:subgauss}
  \|\wh{\mu}_N - \mu\|\le C\left( \sqrt{\frac{\lambda_1\log(1/\delta) }{N}} + \sqrt{\frac{\sum_{i=1}^d \lambda_i}{N}}  \right)~,
\end{equation}
where $C$ is a universal constant and $\lambda_1\ge \lambda_2 \cdots \ge \lambda_d \ge 0$ are the eigenvalues
of the covariance matrix $\Sigma$, see Lugosi and Mendelson
\cite{LuMe16a,LuMe20}, Hopkins \cite{Hop18}, Cherapanamjeri,
Flammarion, and Bartlett \cite{ChFlBa19},
Depersin and Lecu{\'e} \cite{DeLe19},
Lei, Luh, Venkat, and Zhang \cite{LeLuVeZh20},
 Diakonikolas, Kane, Pensia \cite{DiKaPe20}.

The reason for the term ``sub-Gaussian'' is, again, the comparison to what happens in the Gaussian case and the estimator being the empirical mean. Indeed, when the $X_i$ have a multivariate normal distribution and
$\wt{\mu}_N = (1/N)\sum_{i=1}^N X_i$ is the empirical mean, then (\ref{eq:subgauss}) holds with probability at least $1-\delta$. This bound is of the correct order, since the expected value $\EXP \|\wt{\mu}_N - \mu\|$ is proportional to
$\sqrt{\Tr(\Sigma)/N}= \sqrt{(1/N)\sum_{i=1}^d \lambda_i}$, and the term $\sqrt{(1/N)\lambda_1\log(1/\delta)}$ bounds
the fluctuations, using the Gaussian concentration inequality. We refer to the two terms on the right-hand side of
(\ref{eq:subgauss}) as the \emph{weak} and \emph{strong} terms,
respectively:
the strong term is simply the $L_2$ norm of $\|X-\EXP X\|_2$ and the
weak term corresponds to the largest variance of a one dimensional
marginal of $X$, that is, to $\sup_{u \in S^{d-1}} \sigma(u)$.
Here $S^{d-1}=\{x\in \Rd: \|x\|=1\}$ denotes the Euclidean unit sphere and $\inr{\cdot,\cdot}$ is the standard inner product in $\Rd$.

\begin{remark}
Strong-weak inequalities are an important notion in high-dimensional
probability (see, e.g., Lata{\l}a and Wojtaszczyk \cite{MR2449135}). We explain the connection between our results and these inequalities in Appendix \ref{sec:strong-weak}.
\end{remark}

Clearly, an equivalent way of writing the inequality (\ref{eq:subgauss}) is as follows:
\begin{equation}
\label{eq:subgauss1}
  \forall u \in S^{d-1}: \inr{\wh{\mu}_N - \mu, u}\le  C\left( \sqrt{\frac{\lambda_1\log(1/\delta) }{N}} + \sqrt{\frac{\sum_{i=1}^d \lambda_i}{N}} \right)~.
\end{equation}
It is reasonable to expect that \eqref{eq:subgauss1} is the best
``directional formulation" that one can hope for. Firstly, the error
must have a global component, which is the strong term: directional
information corresponds only to one-dimensional marginals of $X$,
while higher dimensional marginals impact the ability of approximating
the mean. At the same time, any standard way of controlling
fluctuations is based on estimates on the worst direction, leading to
the weak term which involves $\lambda_1$. With that in mind, a more
fine-grained version of \eqref{eq:subgauss1} seems unlikely.

Perhaps contrary to intuition, our main result does precisely that: under a mild additional assumption on $X$ we construct an estimator that, up to the optimal strong term, preforms in every direction as if it were an optimal  estimator of the one dimensional marginal:

%

\begin{tcolorbox}
\begin{theorem}
\label{thm:meanest}
Let $X_1,\ldots,X_N$ be i.i.d.\ random vectors, taking values in $\Rd$, with mean $\mu$ and covariance matrix $\Sigma$
whose eigenvalues are $\lambda_1\ge \lambda_2 \cdots \ge \lambda_d \ge 0$.
Suppose that there exists $q>2$ and a constant $\kappa$ such that, for all $u\in S^{d-1}$,
\begin{equation}
\label{eq:normequivalence}
     \left( \EXP \inr{X-\mu, u}^q \right)^{1/q}  \le \kappa      \left( \EXP \inr{X-\mu, u}^2 \right)^{1/2}~.
\end{equation}
Then for every $\delta \in (0,1)$ there exists a mean estimator
$\wh{\mu}_N$ and constants $0<c,c',C < \infty$
(depending on $\kappa$ and $q$ only)
such that, if $\delta \ge e^{-c'N}$, then,
with probability, at least $1-\delta$,
\begin{equation}
\label{eq:subgauss2}
  \forall u \in S^{d-1}: \inr{\wh{\mu}_N - \mu, u}\le  C\left( \sqrt{\frac{\sigma^2(u)\log(1/\delta) }{N}} + \sqrt{\frac{\sum_{i=c\log(1/\delta)}^d \lambda_i}{N}} \right)~.
\end{equation}

\end{theorem}
\end{tcolorbox}

It should be stressed that a proof of such a result, that is sensitive
to directions, calls for the development of a completely new
machinery. In \eqref{eq:subgauss1} one is allowed fluctuations at
scale $\sqrt{\frac{\lambda_1\log(1/\delta) }{N}}$ in all directions,
and that plays a crucial component in the proof of
\cite{LuMe16a,LuMe20}. In contrast, inequality \eqref{eq:subgauss2}
calls for (uniform) control over all nontrivial directions but the
allowed scales of fluctuations in each direction can be much smaller
than $\sqrt{\frac{\lambda_1\log(1/\delta) }{N}}$. That difference
renders useless the methods used in the proof of \eqref{eq:subgauss1}.

The main technical novelty of this article is the machinery that leads
to the necessary directional control.
It is presented in Section \ref{sec:uniform}. This machinery consists of bounds for \emph{ratios} of empirical and
true probabilities that hold uniformly in a class of
functions. Informally put, we are able to control
\[
\sup_{\{f \in \F, \|f\|_{L_2} \geq r\}} \sup_{t: \PROB\{f(X)>t\} \geq \Delta} \left|\frac{N^{-1}\sum_{i=1}^N \IND_{f(X_i)>t}}{\PROB\{f(X)>t\}} -1 \right|
\]
for appropriate values of $r$ and $\Delta$.

In other words, we show that, under minimal assumptions on the class
$\F$, the empirical frequencies of level sets of every $f \in \F$ are
close, in a multiplicative sense, to their true probabilities---as
long as $\|f\|_{L_2}$ and $\PROB\{f(X)>t\}$ are large
enough. Estimates of this flavor have been derived before, but only in
a limited scope. Examples include the classical inequalities of
Vapnik-Chervonenkis in {\sc vc} theory, dealing with small classes of
binary-valued functions (see also, Gin\'{e} and Koltchinskii
\cite{MR2243881} for some results for real-valued classes). Existing
ratio estimates are often based on the highly restrictive assumption
that the collection of level sets,
say of the form $\{ \{x: f(x) > t\} : f \in \F, \ t \geq t_0\}$, is
small in the {\sc vc} sense.
Unfortunately, in the general context that is required here, such an assumption need not hold.

The new method we develop here is based on a completely different argument that builds on the so-called \emph{small-ball method}. The relevant ratio estimate can be found in Theorem \ref{thm:uniform-properties}.

\vskip0.3cm

While the main thrust in \eqref{eq:subgauss2} is the directional dependence, it is worth noting that the strong term is better than $(\E \|X - \EXP X\|_2^2)^{1/2}$. To put this in perspective, consider again the example when the distribution of the data is Gaussian and the estimation procedure is the empirical mean $\wt{\mu}_N= (1/N)\sum_{i=1}^N X_i$.
The next observation shows that even in this well-behaved example,
the strong term needs to be at least of the order
\[
\sqrt{\frac{\sum_{i>k}\lambda_i}{N}}
\]
where $k$ is proportional to $\log(1/\delta)$.

\begin{proposition}
\label{prop:empmean}
Let $\wt{\mu}_N= (1/N)\sum_{i=1}^N X_i$ where the $X_i$ are independent Gaussian vectors with mean $\mu$ and
covariance matrix $\Sigma$.
Suppose that there exists a constant $C$ such that, for all $\delta,N,\mu$, and $\Sigma$,
with probability at least $1-\delta$,
\begin{equation}
\label{eq:ass}
\forall u \in S^{d-1}: \inr{\wt{\mu}_N - \mu, u}\le  C \sqrt{\frac{\sigma^2(u)\log(1/\delta) }{N}} + S~.
\end{equation}
Then there exists a constant $C'$ depending on $C$ only, such that the ``strong term'' $S$ has to satisfy
\[
   S \ge C' \sqrt{\frac{\sum_{i>k_0}\lambda_i}{N}}
\]
where $k_0= 1+ (2C+\sqrt{2})^2\log(1/\delta)$.
\end{proposition}

The proof is given in Section \ref{sec:proofofempmean} of the
Appendix.

\begin{remark}
An easy modification of the proof of Proposition \ref{prop:empmean} reveals that
the lower bound is tight in the sense that, if the $X_i$ are
independent Gaussians, then the empirical mean $\wt{\mu}_N$ satisfies that,
with probability, at least $1-\delta$,
\[
\forall u \in S^{d-1}: \inr{\wt{\mu}_N - \mu, u}\le  C \left(
\sqrt{\frac{\sigma^2(u)\log(1/\delta) }{N}} + \sqrt{\frac{\sum_{i>k_1}\lambda_i}{N}}
\right)~
\]
where $k_1 =c \log(1/\delta)$, for some constants $c$ and $C$.
\end{remark}

The article is organized as follows. In the next two sections we develop the technical machinery we require. First, in Section \ref{sec:tailintegration} we recall --- and appropriately
modify --- the solution of Mendelson \cite{Men20} of the following moment estimation problem: given a real random variable $Z$, find an almost isometric, data dependent estimator of $\EXP |Z|^p$. In other words, find $\wh{Z}_p$ such that, $(1-\eps)\EXP |Z|^p \leq  \wh{Z}_p \leq (1+\eps)\EXP |Z|^p$, with probability $1-\delta$.

The analysis of the moment estimation problem reveals the importance of uniform control of ratios of empirical and
true probabilities, and that is the subject of Section \ref{sec:uniform}
where the new general inequality is proven.

The next component in the proof of Theorem \ref{thm:meanest} is a covariance
estimator (i.e., an estimator of all the directional variances
$\sigma^2(u)$), which we introduce and analyze in Section \ref{sec:variance}.
\emph{Covariance estimation} has received quite a lot of attention
lately, see, for example,
Catoni \cite{Cat16},
Giulini \cite{Giu18}.
Koltchinskii and Lounici \cite{KoLo17},
Lounici \cite{Lou14},
Mendelson \cite{Men18a},
Mendelson and Zhivotovksiy \cite{MeZh18},
Minsker \cite{Min18},
Minsker and Wei \cite{MiWe20}.

The notion of estimation needed here is rather weak: we only require that the
variances are approximated within a multiplicative constant factor.  Our construction is based on the moment estimators developed in Section \ref{sec:tailintegration} and in Section \ref{sec:uniform}:  we construct an estimator $\psi_N(u)$ of the directional variances such that,
with probability at least $1-\delta$, simultaneously for all $u$ such that $\sigma^2(u) \ge
\sum_{i>c\log(1/\delta)}\lambda_i$, all estimated variances
$\psi_N(u)$ satisfy
\[
     \frac{1}{4} \le   \frac{\psi_N(u)}{\sigma^2(u)} \le 2~.
\]
Moreover, for all other $u$, we have $\psi_N(u) \le C
\sum_{i>c\log(1/\delta)}\lambda_i$.

\begin{remark}
It should be noted that covariance estimators, quite different
in nature, but with performance bounds of the same
spirit, were defined by Catoni \cite{Cat16} and
Giulini \cite{Giu18}, under certain fourth-moment assumptions.
\end{remark}

Finally, in Section \ref{sec:meanest} we define the main multivariate mean
estimation procedure and prove Theorem  \ref{thm:meanest}.
Some of the proofs are relegated to the Appendix.


\remark ({\sc nonatomic distributions.})
To avoid unimportant but somewhat tedious technicalities, we assume
throughout that the distribution of the $X_i$ is absolutely continuous
with respect to the Lebesgue measure. This implies that the
distribution of  $\inr{X,u}$ is
nonatomic for all $u\in S^{d-1}$ and we do not need to worry about
multiple points taking the same value---which makes the definition of
trimming and quantiles simpler. This assumption is not restrictive
because one may always add a tiny random perturbation to each data
point, converting the distribution absolutely continuous, and without
changing the mean vector too much.

%

\remark ({\sc Hilbert spaces.})
For convenience, we present our results for random variables taking
values in $\Rd$. However, the main results remain true without
modification when $X$ takes values in any separable Hilbert space ${\cal H}$. The only
condition that needs to be modified is absolute continuity. It
suffices that $\inr{u,X}$ has a continuous distribution for all $u \in
{\cal H} \setminus \{0\}$.


\remark ({\sc computation.})
Our definition of a mean estimator as a (measurable) function of the
data ignores important computational issues. An important branch of
research has focused on computational issues of robust statistical
estimation. In particular, mean estimators achieving sub-Gaussian
performance of the type \eqref{eq:subgauss} that can be computed in
polynomial time have been proposed, see
Hopkins \cite{Hop18}, Cherapanamjeri,
Flammarion, and Bartlett \cite{ChFlBa19},
Depersin and Lecu{\'e} \cite{DeLe19},
Lei, Luh, Venkat, and Zhang \cite{LeLuVeZh20},
Diakonikolas, Kane, Pensia \cite{DiKaPe20}. It is an interesting open
problem whether there exists an efficiently computable estimator
achieving a performance like the one announced in Theorem
\ref{thm:meanest}.

\section{Empirical tail integration and moment estimation}
\label{sec:tailintegration}

In this section we consider the problem of estimating the raw moments
of a real random variable. The estimators studied here are trimmed
estimators, that is, based on the empirical mean after discarding the
smallest and largest values of the sample. Such estimators have been
studied extensively (see Lugosi and Mendelson \cite{LuMe19} for
pointers to the literature). The properties presented in this section have been mostly
developed by Mendelson \cite{Men20} and the exposition below is a slight
modification of the arguments in \cite{Men20}. For completeness, we
present the proofs in Section \ref{sec:proofsoftailintegration}.

Note that the results in this section are for moments of any order
$p\ge 1$, although we only need the special cases $p=1$ and $p=2$.

Let $p\ge 1$ and let $Z$ be a real-valued random variable with continuous distribution
such that $\EXP |Z|^p <\infty$.
Let
$Z_1^N=(Z_1,\ldots,Z_N)$ be a sample of $N$ independent copies of
$Z$.
For fixed $1/N <\theta \leq 1/2$ let
$J_+$ be the set of indices of the $\theta N$ largest values of
$(Z_i)_{i=1}^N$ and denote by $J_-$ the set of indices of the $\theta
N$ smallest values.
(Since $Z$ is assumed to have a continuous distribution, $J_+$ and
$J_-$ are uniquely defined, with probability one.)
Writing $[N]=\{1,\ldots,N\}$, our goal is to show that with high probability,
$$
\Psi_{p,\theta}(Z_1^N) = \frac{1}{N} \sum_{i \in [N] \backslash (J_+ \cup J_-)} |Z_i|^p
$$
is a good estimator of $\E |Z|^p$ for an appropriately chosen value of
$\theta$.

Denote the empirical measure of $Z$ by $\PROB_N$, that is,
for all measurable sets $A\subset \R$, we write
\[
\PROB_N\{Z\in A\} =
     \frac{1}{N} \sum_{i=1}^N \IND_{Z_i\in A}~.
\]
We show that, in order to ensure that $\Psi_{p,\theta}$ is a good
estimator, it suffices to control the ratios $\PROB_N\{Z\in
I\}/\PROB\{Z\in I\}$, for
intervals\footnote{By an \emph{interval} we mean open, closed,
  half-open intervals in $\R$, including rays.} $I$.
Moreover, we show that, under similar assumptions,
$$
\Phi_\theta(Z_1^N)=\frac{1}{N} \sum_{i \in [N] \backslash (J_+ \cup J_-)} Z_i
$$
is a good estimator of the mean $\E Z$.

Our starting point is the following definition that describes some properties that ensure that $\Psi_{p,\theta}$ and $\Phi_
\theta$ are well-behaved, in a sense made precise below.
\begin{tcolorbox}
\begin{definition}
\label{def:ratioconditions}
Set $0 <\Delta,\theta <1/100$. Let ${\cal A}$ be the
event on which the following properties hold:
\begin{description}
\item{$(1)$}
for all nonnegative integers  $j$ and $t \geq 0$ such that
$2^{-j}\PROB\{Z>t\} \geq \Delta$, we have
$$
\left|\frac{\PROB_N\{Z > t\}}{\PROB\{Z>t\} } -1 \right| \leq 2^{-j/2-1}
$$
and if $2^{-j}\PROB\{ Z<-t\} \geq \Delta$, then
$$
\left|\frac{\PROB_N\{Z < - t\}}{\PROB\{Z < -t\}} -1 \right| \leq 2^{-j/2-1}~;
$$
\item{$(2)$} for any interval $I \subset \R$,
$$
\PROB_N\{Z \in I\} \leq \frac{3}{2}\PROB\{Z \in I\} + 2\Delta~;
$$
\item{$(3)$} $\PROB\{Z>0\} \geq \eta$ and $\PROB\{ Z<0 \} \geq \eta$ for $\eta = 4\theta+16\Delta$.
\end{description}
\end{definition}
\end{tcolorbox}


Note that property $(1)$ requires that the relative error of the empirical
measure becomes smaller for larger sets.
The idea behind the conditions of Definition \ref{def:ratioconditions}
is that one may write
$$
\E|Z|^p=\int_0^\infty pt^{p-1} \left(\PROB\{Z>t\}+\PROB\{-Z>t\}\right)dt~,
$$
and with sufficient control of the ratios
$\PROB_N\{Z>t\}/\PROB\{Z>t\}$ and $\PROB_N\{-Z>t\}/\PROB\{-Z>t\}$,
 the integral can be well approximated by an empirical functional like
 $\Psi_{p,\theta}$.
That approximation can be valid even when the distortion is relatively
large for sets $\{Z>t\}$ or $\{-Z>t\}$ whose measure is small,
but a small distortion is essential for sets of relatively large
measure, as such sets have a much higher impact on the integral.

The main technical fact we use is the following estimate on the positive an negative parts of $Z$, denoted throughout by $Z_+$ and $Z_-$, respectively.
For any $\alpha \in (0,1)$, we denote by $Q_{\alpha}$ the
$\alpha$-quantile of the random variable $Z$, that is, $Q_{\alpha}$
is the unique value such that $\PROB\{Z \le Q_{\alpha}\} = \alpha$.

\begin{theorem}
\label{thm:main-integral}
Let $p \geq 1$ and let $Z$ be a real random variable with continuous
distribution such that $\EXP|Z|^p<\infty$.
 Let $0< \Delta,\theta < 1/100$ be such that $\theta N$ is an integer,
 $\theta \geq 7\Delta$, and let ${\cal A}$ be the event defined in Definition \ref{def:ratioconditions}.
Set
$$
\theta_1=2\theta+8\Delta~, \qquad \text{and} \qquad \theta_2=(2\theta-8\Delta)/3~.
$$
Then, on the event ${\cal A}$, we have
\begin{description}
\item{(a)} For every $i \in J_+$, $Z_i \geq 0$ and for every $i \in
  J_-$, $Z_i \leq 0$. In particular, $Z_i=(Z_i)_+$ for all $i\in J_+$,
  and $Z_i=(Z_i)_-$ for all $i\in J_-$.
\item{(b)}
\[
\frac{1}{N}\sum_{j \in [N] \backslash J_+ }  (Z_i)_+^p \leq \E Z_+^p + 2\sqrt{\Delta} \int_0^{Q_{1-\theta_2}} pt^{p-1}\sqrt{\PROB\{Z>t\}}dt~,
\]
and
\[
\frac{1}{N}\sum_{i \in [N] \backslash J_+ } (Z_i)_+^p
\geq
\E Z_+^p - 3\E \left[Z^p \IND_{\{Z \geq Q_{1-\theta_1}\}} \right]
-  2\sqrt{\Delta} \int_0^{Q_{1-\theta_1}} 2t \sqrt{\PROB\{Z>t\}}dt~.
\]
\item{(c)} The analogous claim to (b) holds for $Z_-$, the negative part of $Z$.
\end{description}
\end{theorem}

The proof of Theorem \ref{thm:main-integral} is given in Section
\ref{sec:proofsoftailintegration}.
We remark here that, as it is shown in the proof, $Q_{1-\theta_1},
Q_{1-\theta_2}>0$ on the event ${\cal A}$.

In order to estimate the $p$-th moment of a random variable $Z$, we
simply write
\begin{eqnarray*}
\frac{1}{N} \sum_{i \in [N] \backslash (J_+ \cup J_-)} |Z_i|^p
& = &\frac{1}{N} \sum_{i \in [N] \backslash (J_+ \cup J_-)} (Z_i)_+^p + \frac{1}{N} \sum_{i \in [N] \backslash (J_+ \cup J_-)} (Z_i)_-^p
\\
& = & \frac{1}{N} \sum_{i \in [N] \backslash J_+}  (Z_i)_+^p + \frac{1}{N} \sum_{i \in [N] \backslash  J_-} (Z_i)_-^p~,
\end{eqnarray*}
where the last equality holds on the event ${\cal A}$ because of
part (a) of Theorem \ref{thm:main-integral}.
Now part (b) may be used to show that the two
terms on the right-hand side are close to their means $\E Z_+^p$ and $\E Z_-^p$.

We show in Section \ref{sec:uniform} that properties $(1)-(2)$ in
Definition \ref{def:ratioconditions} are satisfied with high
probability.
Property $(3)$ means that the random variable $Z$ is sufficiently
``balanced".

Next we derive a corollary of Theorem \ref{thm:main-integral} that is
more convenient to use. To this end, set
\[
{\cal E}_{T,p} = 2 \sqrt{\Delta}\int_0^T pt^{p-1} \sqrt{\PROB\{|Z|>t\}}dt~.
\]

The following lemma was established by Mendelson \cite{Men20} for
nonnegative-valued, absolutely continuous random variables. The
extension to the case below (allowing an atom at zero) is
straightforward; the proof is omitted.

\begin{lemma}
\label{lemma:basic-parameters}
There is an absolute constant $c$ for which the following holds.
Let $p\ge 1$ and $\kappa \in (0,1)$.
Let $Z$ be nonnegative random variable whose distribution is a mixture of a an
absolutely continuous component and an atom at $0$, such that $\EXP
Z^{2p} < \infty$. Then
$$
\E \left[ Z^p \IND_{\{Z > Q_{1-\kappa}\}} \right] \leq \sqrt{\kappa} \left(\EXP Z^{2p}\right)^{1/2}~,
$$
and
$$
{\cal E}_{Q_{1-\kappa},p} \leq c\sqrt{\Delta} \sqrt{\log\left(\frac{1}{\kappa}\right)} \left(\EXP Z^{2p}\right)^{1/2}~.
$$
Moreover, if $\EXP Z^q < \infty$ for some $q>2p$, then
$$
{\cal E}_{Q_{1-\kappa},p} \leq c_{q,p}\sqrt{\Delta}  \left(\EXP Z^q\right)^{p/q}~,
$$
where $c_{q,p} = c p/(q-2p)$ for a numerical constant $c>0$.
\end{lemma}

Combining Lemma \ref{lemma:basic-parameters} with Theorem
\ref{thm:main-integral} leads to the following
corollary:

\begin{corollary} \label{thm:main-single-integral}
There are absolute constants $c_1,\ldots,c_4$ for which the following
holds.
Assume the conditions of Theorem \ref{thm:main-integral}.
Set $c_1\frac{\log N}{N} \leq \Delta <100$, and $\theta = c_2 \Delta$ with
$c_2>7$. Then
\[
\left|\frac{1}{N}\sum_{i \in [N] \backslash J_+} (Z_i)_+^p - \E Z_+^p \right| \leq  c_3 \sqrt{\Delta \log\left(\frac{1}{\Delta}\right)} \left(\EXP Z^{2p}\right)^{1/2}~,
\]
and if $Z \in L_q$ for $q > 2p$ then
\[
\left|\frac{1}{N}\sum_{i \in [N] \backslash J_+} f_+^p(X_i) - \E
  f_+^p\right| \leq c_{q,p} \sqrt{\Delta}  \left(\EXP Z^q\right)^{p/q}
\]
where $c_{q,p} = c_4 p/(q-2p)$.

The analogous inequalities hold for $Z_-$ with $J_-$ replacing $J_+$.
\end{corollary}

\subsection{Proof of Theorem \ref{thm:main-integral}}
\label{sec:proofsoftailintegration}

The proof is a minor modification of some arguments of Mendelson
\cite{Men20}.

The proof of Theorem \ref{thm:main-integral} requires a few preliminary steps. First, observe that $\theta_2 \geq 2\Delta$ and therefore
\begin{equation} \label{eq:observation1}
\text{for all} \ t \leq Q_{1-\theta_2}~, \quad \PROB\{ Z>t\} \geq 2\Delta~.
\end{equation}
\eqref{eq:observation1} implies that all the level sets $\{Z>t\}$ for
$0<t \leq Q_{1-\theta_2}$ satisfy property $(1)$ of Definition
\ref{def:ratioconditions} with $j=1$,  a fact used frequently in what follows.

Another useful observation is that by Property $(3)$,
$$
\PROB\{Z>0\} \geq \gamma = 4\theta+16\Delta =2\theta_1~,
$$
and therefore,
the $(1-\theta_1)$-quantile of $Z_+$ coincides with the
$(1-\theta_1)$-quantile of $Z$ and the $(1-\theta_1)$-quantile of $Z_-$
equals $- Q_{\theta_1}$.

Moreover, by property $(1)$ (with $j=0$) and the choice of $\theta$ we have
$$
\PROB_N\{Z>0\} \geq \frac{1}{2}\PROB\{Z>0\} \geq  \frac{\theta_1}{2} > \theta
$$
and an identical argument shows that $\PROB_N\{Z<0\} > \theta$.
Now let $(Z_i^\sharp)_{i=1}^N$ be the monotone nonincreasing
rearrangement of $(Z_i)_{i=1}^N$.
Set $\wh{Q}_+ = Z_{\theta N}^\sharp$ and $\wh{Q}_- = Z_{(1-\theta)  N}^\sharp$.
Clearly,  $\wh{Q}_+>0$ and for  $i\in J_+$,  $Z_i=(Z_i)_+$. The
analogous statement holds for $\wh{Q}_-$ and $J_-$, proving part (a) of Theorem \ref{thm:main-integral}.

\begin{lemma} \label{lemma:est-on-hat-Q}
On the event ${\cal A}$ defined in Definition \ref{def:ratioconditions},
$$
Q_{1-\theta_1} < \wh{Q}_+ < Q_{1-\theta_2} \ \ {\rm and} \ \ -Q_{\theta_1} < \wh{Q}_- < - Q_{\theta_2}~.
$$
\end{lemma}

\proof We present a proof of the first claim. The proof of the second one is identical and is omitted.
By definition,
$\PROB_N \{Z \geq \wh{Q}_+\} = \theta$.
Applying property $(2)$  in Definition \ref{def:ratioconditions} for $I=[\wh{Q}_+,\infty)$,
$$
\theta \leq \PROB_N\{ Z \geq \wh{Q}_+\} \leq \frac{3}{2}\PROB\{ Z \geq \wh{Q}\} + 2\Delta~.
$$
In particular,
$$
\PROB\{ Z > \wh{Q}_+\}=\PROB\{Z \geq \wh{Q}_+\} \geq \frac{2}{3}\left(\theta - 2\Delta\right) \geq \Delta
$$
provided that $\theta \geq 7\Delta/2$, as was assumed. Hence, we may
use property $(1)$ of Definition \ref{def:ratioconditions} with $j=1$
and $t= \wh{Q}_+$, to get
\[
\left|\frac{\PROB_N\{ Z > \wh{Q}_+\}}{\PROB\{Z  > \wh{Q}_+\}}-1 \right| \leq \frac{1}{2}~.
\]
This implies
$$
\PROB\{ Z > \wh{Q}_+\} \leq 2 \PROB_N\{ Z > \wh{Q}_+\} \leq 2\theta < \theta_1
$$
and
$$
\PROB\{ Z > \wh{Q}_+\}  \geq \frac{2\PROB_N\{Z > \wh{Q}_+\}}{3} =
\frac{2(\theta-1/N)}{3} \ge \theta_2~,
$$
and therefore
\[
Q_{1-\theta_1} < \wh{Q}_+ < Q_{1-\theta_2}~,
\]
as claimed.
\endproof

\begin{lemma}
\label{lemma:int-preliminary}
Consider the conditions of Theorem \ref{thm:main-integral}. Then, on
the event ${\cal A}$,
\[
\int_0^{Q_{1-\theta_1}} pt^{p-1}\PROB_N\{Z>t\} dt - \theta \wh{Q}_+^p
\leq \frac{1}{N}\sum_{j \in [N] \backslash J_+} (Z_i)_+^p
\leq \int_0^{Q_{1-\theta_2}} pt^{p-1}\PROB_N\{Z >t\} dt~.
\]
A similar estimate holds for $Z_-$.
\end{lemma}

\proof Recall that $\wh{Q}_+ = Z_{\theta N}^\sharp$, and therefore,
\[
\frac{1}{N}\sum_{i=1}^N (Z_i)_+^p\IND_{\{(Z_i)_+ \leq \wh{Q}_+\}} - \theta \wh{Q}_+^p
\leq \frac{1}{N}\sum_{i \in [N] \backslash J_+} (Z_i)_+^p
\leq \frac{1}{N}\sum_{i=1}^N (Z_i)_+^p\IND_{\{(Z_i)_+ \leq \wh{Q}_+\}}~.
\]
By tail integration,
\[
\frac{1}{N}\sum_{i=1}^N (Z_i)_+^p\IND_{\{(Z_i)_+ \leq \wh{Q}_+\}}
= \int_0^\infty pt^{p-1} \PROB_N \left\{ (Z)_+^p\IND_{\{(Z)_+ \leq
    \wh{Q}_+\}}  > t\right\} dt
= \int_0^{\wh{Q}_+} pt^{p-1}\PROB_N \{ Z > t\} dt~.
\]
Since by Lemma \ref{lemma:est-on-hat-Q}, on the event ${\cal A}$,  $Q_{1-\theta_1} < \wh{Q}_+ <
Q_{1-\theta_2}$, the claimed inequalities follow.
\endproof

With Lemma \ref{lemma:int-preliminary} in mind, next we
may use a general estimate of Mendelson \cite{Men20} for $\int_0^T pt^{p-1} \PROB_N\{ Z>t\} dt$ that
holds as long as $T$ is such that  $\PROB\{Z>t\}$ is large enough. To formulate the claim, recall that
$$
{\cal E}_{T,p} = 2 \sqrt{\Delta}\int_0^T pt^{p-1} \sqrt{\PROB\{|Z|>t\}}dt~.
$$

\begin{lemma} \label{lemma:integral-1}
(Mendelson \cite{Men20}.)
Let $T$ be such that $\PROB\{ Z >T\} \geq \Delta$. On the event
${\cal  A}$ of Definition \ref{def:ratioconditions}, we have
\[
\E \left[ Z_+^p\IND_{\{Z_+ \leq T\}} \right] - {\cal E}_{T,p} \leq
\int_0^T pt^{p-1} \PROB_N\{Z >t\} dt
\leq \E Z_+^p + {\cal E}_{T,p}~.
\]
\end{lemma}

Now we are ready to prove Theorem \ref{thm:main-integral}.

\noindent{\bf Proof of Theorem \ref{thm:main-integral}.} \
Assume that the event ${\cal A}$ holds.
Apply Lemma \ref{lemma:integral-1} with $T= Q_{1-\theta_1}$ and $T=
Q_{1-\theta_2}$. Both valid choices, as
\[
\PROB\{Z > Q_{1-\theta_1} \} \geq \PROB\{ Z > Q_{1-\theta_2} \} \geq \Delta~.
\]
Thus,
\[
\int_0^{Q_{1-\theta_2}} pt^{p-1} \PROB_N \{Z  >t\}  dt \leq \E Z_+^p + {\cal E}_{Q_{1-\theta_2},p}~,
\]
and
\begin{eqnarray*}
\int_0^{Q_{1-\theta_1}}  pt^{p-1}\PROB_N \{Z >t\} dt
& \geq & \E \left[ Z_+^p \IND_{\{Z \leq Q_{1-\theta_1}\}} \right] - {\cal E}_{Q_{1-\theta_1},p}
\\
& = & \E Z_+^p - \left( \E \left[ Z^p \IND_{\{Z > Q_{1-\theta_1}\}} \right]
         +{\cal E}_{Q_{1-\theta_1},p}\right)~.
\end{eqnarray*}

It remains to show that
\begin{equation} \label{eq:hat-Q-tail}
\wh{Q}_+^p \theta \leq 2 \E \left[ Z^p\IND_{\{Z \geq
    Q_{1-\theta_1}\}} \right]~,
\end{equation}
which, by Lemma \ref{lemma:int-preliminary} completes the proof.
To that end, recall that $\PROB\{Z >Q_{1-\theta_2}\} \geq \Delta$ and
that, by Lemma \ref{lemma:est-on-hat-Q},
$Q_{1-\theta_1} < \wh{Q}_+ < Q_{1-\theta_2}$. Thus,
\[
\E \left[ Z^p \IND_{\{Z \geq \wh{Q}_+\}} | Z_1,\ldots,Z_N\right]  \leq
\E \left[ Z_+^p \IND_{\{Z \geq Q_{1-\theta_1}\}} \right].
\]
Also, since $\PROB_N\{Z \geq \wh{Q}_+\} \geq \theta$ and
$$
\PROB\{Z \geq \wh{Q}_+ |Z_1,\ldots,Z_N\}  \geq \PROB\{Z  \geq Q_{1-\theta_2}\} \geq \Delta~,
$$
it follows from property $(1)$ in Definition \ref{def:ratioconditions}
(taking $j=0$) that
\[
\E \left[ Z^p\IND_{\{Z \geq \wh{Q}_+\}} |Z_1,\ldots,Z_N\right]
\geq \wh{Q}_+^p \PROB\{Z \geq \wh{Q}_+ |Z_1,\ldots,Z_N\}
\geq \wh{Q}_+^p 2 \PROB_N\{Z \geq \wh{Q}_+\} = 2 \wh{Q}_+^p \theta~,
\]
proving \eqref{eq:hat-Q-tail}.

A similar estimate holds for $Z_-$.
\endproof

\section{Uniform relative deviations of empirical measures}
\label{sec:uniform}

In this section we present inequalities for uniform relative
deviations of empirical measures. These are the main technical
novelties of the article that allow us to construct covariance and mean
estimators with the desired properties. More precisely, we prove that
properties (1) and (2) of Definition \ref{def:ratioconditions} hold
\emph{uniformly} over a class of random variables (with high probability) above a certain
critical level depending on the class.

To properly set up the main result of this section, we need a few
definitions. Since the results may be of independent interest, we
present it in a greater generality than what is needed for our
purposes in this article.

Let $\X$ be a measurable set and let $X,X_1,\ldots,X_N$ be
independent, identically distributed random variables taking values in $\X$.

Consider a class $\F$ of real-valued functions defined on $\X$. We
assume that, for all $f\in \F$, the random variable $f(X)$ has mean
zero.
We denote by $L_2$ the set of functions $f:\X\to \R$ such that
$\EXP f(X)^2 < \infty$. We write $\|f\|_{L_2} = \left(\EXP
  f(X)^2\right)^{1/2}$ for all $f\in L_2$ and denote the unit sphere
and the unit ball in $L_2$ by $S(L_2)= \{f\in L_2: \|f\|_{L_2}=1\}$
and $D= \{f\in L_2: \|f\|_{L_2} \le 1\}$, respectively.

Recall that for an indicator function $h(x)=\IND_{x\in A}$ for some $A \subset
\X$, we abbreviate
\[
    \PROB\{ h\}= \PROB\{ X \in A \} \quad \text{and} \quad
    \PROB_N\{ h\} = \frac{1}{N} \sum_{i=1}^N \IND_{X_i \in A}~.
\]

 We assume that
$0\in \F$ (i.e., $\F$ contains the function that equals zero
everywhere) and that $\F$ is star-shaped around $0$, that is,
if $f\in \F$ then $cf \in \F$ for all $c\in [0,1]$.

For $\epsilon >0$, denote by ${\cal M}(\F,\epsilon)$ the packing
number of $\F$, that is, the size of the largest $\epsilon$-net in
$\F$ (i.e., a subset whose elements have pairwise distance at least $\epsilon$).

\begin{definition}
\label{def:criticallevel}
For $\Delta,c,L>0$, let
$\ol{\rho}_N=\ol{\rho}_N(c,\Delta,L)$ denote the infimum of all those values $r>0$
such that
\begin{align*}
(1) \ & \ \log {\cal M}\left(\F \cap r S(L_2),  \frac{\Delta^{3/2}
      r}{1000 \cdot L} \right) \leq c N \Delta~, \quad \text{ and}   \\
(2) \ &  \ \E \sup_{u \in (\F-\F) \cap \Delta^{3/2} r D}
  \left|\frac{1}{N} \sum_{i=1}^N \eps_i u(X_i) \right| \leq
  \frac{\Delta^2 r}{40000 \cdot L}  ~,
\end{align*}
where $\epsilon_1,\ldots,\epsilon_N$ are independent symmetric
Bernoulli random variables with $\PROB\{\epsilon_i=1\}
=\PROB\{\epsilon_i=- 1\} =1/2$ and $\F-\F=\{f=g-h: g,h\in \F\}$.
We call $\ol{\rho}_N$ the \emph{critical level} of the class $\F$.
\end{definition}

Note that since $\F$ is star-shaped, inequalities $(1)$ and $(2)$ in
Definition \ref{def:criticallevel}  are satisfied for all $r>
\ol{\rho}_N$.
The constants $1/1000$ and $1/40000$ do not have any special role. Their
values have not been optimized and they are chosen by convenience.

The main result of this section is that there exists a positive
numerical constant $c$ such that, above the critical level
$\ol{\rho}_N(c,\Delta,L)$,
functions in $\F$ simultaneously satisfy properties $(1)$ and $(2)$ of
Definition \ref{def:ratioconditions}, given a certain condition
involving the constant  $L$.

To formulate the theorem, for $r>0$, define the classes $U_r$ and
$V_r$ of indicator functions by
\[
U_r = \left\{ \IND_{\{ f \in I\}} : I \ \text{is an interval,} \ f \in
  \F, \ \|f\|_{L_2} \geq r \right\}
\]
and
\[
V_r = \left\{ \IND_{\{ h > t\}} : t >0, \ h \in \F \cup (-\F), \ \|h\|_{L_2} \geq r\right\}~.
\]
The main result of this section is the following:

\begin{tcolorbox}
\begin{theorem} \label{thm:uniform-properties}
Let $\F \subset L_2$ be a class of real-valued functions that is star-shaped
about $0$, such that
$\EXP f(X) =0$ for all $f\in \F$.
There exist positive numerical constants $c,c_0,c_1$ such that the
following holds.
Let $c_0\frac{\log N}{N} \leq \Delta \leq \frac{1}{2}$ and assume that functions in
$\F$ satisfy the following \emph{small-ball condition} with constants $L>0$ and $\gamma = c_1 \Delta$: for any interval $I \subset \R$,
\begin{equation} \label{eq:SB-used}
\PROB\{ f(X) \in I\} \leq \max\left\{\frac{L|I|}{\|f\|_{L_2}}, \gamma\right\}~.
\end{equation}
Suppose that $r > \ol{\rho}_N(c,\Delta,L)$. Then, with probability at least
$1-2\exp(-c_2(L) \Delta N)$, for all $u \in U_r$ and $v \in V_r$,
\begin{description}
\item{(a)} for any integer $j \ge 0$, if  $2^{-j}\PROB\{ v\} \geq \Delta$, then
$$
\left|\frac{\PROB_N \{v\}}{\PROB\{v\}} -1 \right| \leq 2^{-j/2 -1}~;
$$
\item{(b)} $\PROB_N\{u\} \leq \frac{3}{2}\PROB\{u\} + 2\Delta$~,
\end{description}
where $c_2(L)$ is a positive constant depending on $L$ only.
\end{theorem}
\end{tcolorbox}

Theorem \ref{thm:uniform-properties} resembles classical results in
empirical processes theory. In fact, in certain special situations, the classes $U_r$ and
$V_r$ have a well-behaved {\sc vc} dimension  and then uniform ratio estimates are known (see, for example,
\cite{Men20} for a recent application). However, in the general case we study here, the {\sc vc}
dimensions of $U_r$ and $V_r$ can be very large or even
infinite---even when the class $\F$ itself is relatively
well-behaved. As a result, the proof of Theorem
\ref{thm:uniform-properties} calls for a different way of showing that
$U_r$ and $V_r$ are ``small". The key feature of $\F$ used in the
proof of Theorem \ref{thm:uniform-properties} is that functions in
$\F$ satisfy the small-ball condition \eqref{eq:SB-used}.

We begin by showing that any fixed function $f$ satisfies
Properties $(a)$ and $(b)$, with
high probability. This observation, while interesting on its own
right, is needed in the proof of Theorem
\ref{thm:uniform-properties}. As it happens, the required bounds for a
single function do follow from {\sc vc} theory.

\begin{definition} \label{def:vc}
Let $H$ be a class of $\{0,1\}$-valued functions on $\X$. A set $\{x_1,\ldots,x_n\}$ is shattered by $H$ if for every $I \subset [n]$ there is some $h_I \in H$ for which $h_I(x_i)=1$ if $i \in I$ and $h_I(x_i)=0$ otherwise.
The {\sc vc} dimension of $H$ is the maximal cardinality of a subset of $\X$ that is shattered by $H$. It is denoted by ${\rm VC}(H)$.
\end{definition}

We refer the reader to van der Vaart and Wellner \cite{vaWe96} for basic facts on {\sc vc} classes and on the {\sc vc}-dimension.

The connection between the {\sc vc}-dimension and properties $(a)$ and
$(b)$ for a single function is that for any function $f$, the class of indicator functions
\[
U_f = \left\{ \IND_{\{f \in I\}} : I \ \text{is \ an interval} \right\}
\]
satisfies that $VC(U_f) \leq 2$.

For classes of sets with finite {\sc vc} dimension, the following
analogue of Theorem \ref{thm:uniform-properties} was established by
Mendelson \cite{Men20}. Again, it should be stressed that in such cases, uniform ratio estimates are far simpler than in the general scenario that is needed in what follows. 

\begin{theorem} \label{thm:Properties-via-VC}
(\cite{Men20}.)
There are absolute constants $c_0$ and $c_1$ for which the following
holds.
Let $U$ be a class of functions on $\X$ taking values in $\{0,1\}$, such that $VC(U) \leq d$. Then for
\[
c_0 \frac{d}{N} \log\left(\frac{eN}{d}\right) \leq \Delta \leq \frac{1}{2}~,
\]
with probability at least $1-2\exp(-c_1 \Delta N)$, for every $u \in
U$ and nonnegative integer $j$,
\begin{description}
\item{(a')} if $\PROB\{u\} \geq 2^j \Delta$, then
$$
\left|\frac{\PROB_N\{u\}}{\PROB\{u\}} -1 \right| \leq 2^{-j/2-2}~;
$$
\item{(b')} $\PROB_N\{u\} \leq \frac{3}{2}\PROB\{u\} + 2 \Delta$~.
\end{description}
\end{theorem}

In particular, when applied to the class $U_f$ whose VC dimension is
at most $2$, Theorem \ref{thm:Properties-via-VC} implies that any
function $f$ satisfies properties $(a)$ and $(b)$ of Theorem
\ref{thm:uniform-properties}. Indeed, $VC(U_f) \leq 2$ and therefore,
Theorem \ref{thm:Properties-via-VC} may be applied when $\Delta \geq
c_0\frac{\log N}{N}$ for a numerical constant $c_0>0$.

\subsection{Proof of Theorem \ref{thm:uniform-properties}}

Let us now turn to the proof of Theorem
\ref{thm:uniform-properties}. It is important to note once again
that there is no reason to expect that the corresponding classes of
indicator functions $U_r$ and $V_r$ have a well-behaved {\sc vc} dimension.

We begin with the following straightforward observation.

\begin{lemma} \label{lemma:simple-obs}
For any $f,h$, any $t \in \R$ and $\delta>0$,
\begin{equation} \label{eq:simple-obs-diff}
\left|\IND_{\{f>t\}} - \IND_{\{h>t+\delta\}}\right| \leq \IND_{\{|f-h| > \delta\}} + \IND_{\{h \in (t-\delta,t+\delta]\}}.
\end{equation}
\end{lemma}

\proof If $\IND_{\{f>t\}}(x)-\IND_{\{h>t+\delta\}}(x) \not= 0$ and $|h-f|(x) \leq \delta$ then one of the two alternatives holds: either $f(x)>t$ and $h(x) \leq t+\delta$, or $f(x) \leq t$ and $h(x)>t+\delta$. The former implies that $h(x) \in (t-\delta,t+\delta]$, while the latter is impossible.
\endproof

Thanks to Lemma \ref{lemma:simple-obs}, it is possible to estimate
\[
\left|\PROB\{f>t\}-\PROB\{h>t+\delta\}\right| \ \ {\rm and} \ \ \left|\PROB_N\{f>t\}-\PROB_N\{h>t+\delta\}\right|
\]
by a combination of a tail estimate for $|f-h|$ and a small-ball estimate for $|h|$ --- first with respect to the underlying measure $\PROB$ and then with respect to the empirical measure $\PROB_N$.


Fix $\Delta> c_0\log N/N$ for the numerical constant $c_0$ mentioned
in the paragraph following Theorem \ref{thm:Properties-via-VC}.
Let $c>0$ be a constant to be specified later
and  let $r> \ol{\rho}_N(c,\Delta,L)$ be above the critical level for $\F$.
Let $j$ be a nonnegative integer. By property (1) in Definition
\ref{def:criticallevel}, a maximal $\Delta^{3/2}r /(1000 L)$-net of the
subset of $\F$ consisting
of functions with $\|f\|_{L_2}=r$
has cardinality at most $\exp(cN\Delta)$. Let $H_r$ be such an $\Delta^{3/2}r /(1000 L)$-net.

If $c \le c_1/2$ for the constant $c_1$ appearing in Theorem
\ref{thm:Properties-via-VC}, then
by the ratio estimate for a single function (Theorem
\ref{thm:Properties-via-VC}) and the union bound, with probability at least $1-2\exp(-c' N \Delta)$,
\begin{equation}
\label{eq:onthenet}
\sup_{h \in H_r} \sup_{t: \PROB\{h>t\} \geq 2^{j-1}\Delta} \left|\frac{\PROB_N\{h>t\}}{\PROB\{h>t\}}-1\right| \leq \frac{2^{-{(j-1)/2}}}{4}~,
\end{equation}
where we may take $c'=c_1/2$.

For
$f \in \F \cap r S(L_2)$, let $\pi f \in H_r$ be the best
approximation to $f$ in the net $H_r$ with respect to the $L_2$ norm.
Then, for any $t \in \R$ and $\delta>0$, on the same event where
\eqref{eq:onthenet} holds, we have
\begin{equation} \label{eq:ratio-in-proof-1}
\left|\frac{\PROB_N\{f >t\}}{\PROB\{f>t\}} -1\right|
   \leq \left|\frac{\PROB_N\{f >t\}}{\PROB\{f>t\}} -
     \frac{\PROB_N\{\pi f >t+\delta\}}{\PROB\{\pi f >t+\delta\}} \right| +\frac{2^{-{(j-1)/2}}}{4}~,
\end{equation}
provided that $\PROB\{ \pi f >t+\delta\} \geq 2^{(j-1)}\Delta$.

Note that in the inequality above, we may choose the value of
$\delta$ at will, even depending on $f$.
For each $f\in \F$, define $\delta_f  = \Delta\|f\|_{L_2}/(100\cdot L)$

The next lemma shows that with this choice of $\delta_f$, one indeed
has $\PROB\{ \pi f >t+\delta\} \geq 2^{(j-1)}\Delta$ whenever
$\PROB\{ f> t\} \ge 2^j \Delta$.

\begin{lemma} \label{lemma:real-est-1}
Assume the small-ball condition \eqref{eq:SB-used} where $\gamma \le \Delta/18$.
Then, for every $f \in \F \cap r S(L_2)$,
\[
\left|\PROB\{f >t\}-\PROB\{\pi f >t+\delta_f\}\right| \leq \frac{2\Delta}{9}~.
\]
\end{lemma}

\proof Fix $f \in \F \cap r S(L_2)$. By the small-ball condition, if $2L\delta_f/r \geq \gamma$ then
\[
\PROB\{\pi f \in [t-\delta_f,t+\delta_f]\} \leq \frac{2L\delta_f}{\|\pi f\|_{L_2}}=\frac{2L\delta_f}{r}~.
\]
Therefore,
\begin{eqnarray*}
\PROB\{|f-\pi f| \geq \delta_f\} + \PROB\{\pi f \in [t-\delta_f,t+\delta_f]\}
& \leq & \frac{\|f-\pi f\|_{L_2}^2}{\delta_f^2} + \frac{2L\delta_f}{\|\pi
  f\|_{L_2}}   \\
& \leq & \frac{\left(\frac{\Delta^{3/2}r}{1000 \cdot
         L}\right)^2}{\left(\frac{\Delta r}{100 \cdot L}\right)^2}
+ \frac{2L\Delta }{100\cdot L} \\
& = & \frac{3\Delta}{100}~.
\end{eqnarray*}
The stated inequality now follows from Lemma \ref{lemma:simple-obs}.
\endproof

Hence, under the small-ball condition, \eqref{eq:ratio-in-proof-1}
indeed holds whenever $\PROB\{ f> t\} \ge 2^j \Delta$.
Consider such an a function $f\in \F \cap r S(L_2)$.
Using
\[
\left| \frac{\PROB_N\{ \pi f >t+\delta_f\}}{\PROB\{\pi f >t+\delta_f\}}-1\right| \leq
\frac{2^{-{(j-1)/2}}}{4} < \frac{1}{2}~,
\]
the first term on the right-hand side of \eqref{eq:ratio-in-proof-1}
may be bounded as
\begin{eqnarray} \label{eq:pg4}
\lefteqn{
 \left|\frac{\PROB_N\{f >t\}}{\PROB\{f>t\}} - \frac{\PROB_N\{\pi f
  >t+\delta_f\}}{\PROB\{\pi f >t+\delta_f\}}\right| \nonumber }
\\
& \leq &
\left|\frac{\PROB_N\{f >t\}-\PROB_N\{\pi f
         >t+\delta_f\}}{\PROB\{f>t\}}\right| + \frac{\PROB_N\{\pi f
         >t+\delta_f\}}{\PROB\{\pi f >t+\delta_f\}} \left|\frac{\PROB\{f
         >t\}-\PROB(\pi f  >t+\delta_f)}{\PROB\{f>t\}}\right| \nonumber
\\
& \leq & \left|\frac{\PROB_N\{f >t\}-\PROB_N\{ \pi f
         >t+\delta_f\}}{\PROB\{f>t\}}\right| + 2 \left|\frac{\PROB\{f
         >t\}-\PROB\{\pi >t+\delta_f\}}{\PROB\{f>t\}}\right| \nonumber
  \\
& \le &
\frac{\PROB_N\{ |f-\pi f| \geq \delta_f\} }{\PROB\{f>t\}}
+
\frac{\PROB_N\{ \pi f \in [t-\delta_f,t+\delta_f]\}}{\PROB\{f>t\}}
\nonumber \\
& & +
2 \frac{\PROB\{ |f-\pi f| \geq \delta_f\} }{\PROB\{f>t\}}
+
2 \frac{\PROB\{ \pi f \in [t-\delta_f,t+\delta_f]\}}{\PROB\{f>t\}}
\nonumber \\
& & \text{(using Lemma \ref{lemma:simple-obs} twice)}
\nonumber \\
& \defeq & (I) + (II) + (III) + (IV)~.
\end{eqnarray}
Hence, we need to bound the four terms on the right-hand side. The
last two terms may be bounded without further work, as we have already
seen in the proof of Lemma \ref{lemma:real-est-1} that
\[
     (III) + (IV) \le \frac{3\Delta/100}{\PROB\{f>t\}}  \le
     \frac{3\cdot 2^{-j}}{100}~.
\]
In the remaining part of the proof we bound the empirical counterpart,
that is, the terms $(I)$ and $(II)$.

Since $\pi f \in H_r$, for term $(II)$, we may simply invoke part (b')
of Theorem \ref{thm:Properties-via-VC}  that implies that, with
probability at least $1-e^{-c \Delta N}$, for all $f \in \F \cap r S
(L_2)$, we have
\[
     (II) \le \frac{3}{4} \ (IV) + \frac{\Delta/10}{\PROB\{f>t\}}  \le
      2^{-j-3}~.
\]
It remains to bound
\begin{equation}
\label{eq:uniftailestimate}
\sup_{f \in \F \cap r S(L_2)} \PROB_N\{|f-\pi f| > \delta_f\}
= \sup_{f \in \F \cap r S(L_2)} \frac{1}{N}\sum_{i=1}^N \IND_{\{|f-\pi f| > \delta_f\}}(X_i)~.
\end{equation}
This is done in the next lemma that implies that, with probability at
least $1-e^{-c\Delta N}$, for all $f\in \F \cap r S(L_2)$ with
$\PROB\{f>t\} \ge 2^j \Delta$,
\[
   (I) \le \frac{ 2^{-j} }{10}~.
\]
Putting everything together, we get that there exists a universal
constant $c>0$ such that, for every nonnegative integer $j$, with
probability at least $1-e^{-c \Delta N}$, we have
\[
\left|\frac{\PROB_N\{f >t\}}{\PROB\{f>t\}} -1\right|
   \leq   2^{-j} \left( \frac{3 }{100} +\frac{1}{8} + \frac{1}{10} \right)
+\frac{2^{-{(j-1)/2}}}{4} \le
2^{-(j/2-1)}~.
\]
Since there are at most $\log_2 N$ relevant values of $j$, the union
bound implies part (a) of Theorem \ref{thm:uniform-properties}
for functions  $f\in \F \cap r S(L_2)$.
To deal with functions that satisfy $\|f\|_{L_2} \geq r$, fix such a
function and $t \in \R$
for which $\PROB\{f>t\} \geq 2^j \Delta$. Put $t_r=rt/\|f\|_{L_2}$ and $f_r=rf/\|f\|_{L_2} \in F \cap r S(L_2)$, and note that
\[
\{f > t \} = \left\{\frac{rf}{\|f\|_{L_2}} \cdot \frac{\|f\|_{L_2}}{r} > t \right\} = \left\{f_r > t_r\right\}~.
\]
Thus, $\PROB\{f_r > t_r\} \geq 2^j\Delta$,
\[
\frac{\PROB_N\{f_r>t_r\}}{\PROB\{f_r>t_r\}}=\frac{\PROB_N\{f>t\}}{\PROB\{  f>t\}},
\]
and the claim follows from the bound for $\F \cap r S(L_2)$ and by the
star-shaped property of $\F$.

\begin{lemma} \label{thm:empirical-part-1}
Let $r > \ol{\rho}_N(c,\Delta,L)$. Then, for  some constant $c_1>0$, with probability at least
$1-e^{c_1 \Delta N}$,
\[
\sup_{f \in \F \cap r S(L_2)} \frac{1}{N}\sum_{i=1}^N
  \IND_{\{|f-\pi f| > \delta_f\}}(X_i) \leq  \frac{\Delta}{10}~.
\]
\end{lemma}



\proof
Recall that, by definition, $r> \ol{\rho}_N(c,\Delta,L)$ implies that
\begin{equation}
\label{eq:rademacher}
\E \sup_{f \in \F \cap r S(L_2)} \left|\frac{1}{N}\sum_{i=1}^N \eps_i
  (f-\pi f)(X_i) \right| \leq  \frac{\Delta^2 r}{40000 \cdot L}~.
\end{equation}
Define
\[
Z \defeq \sup_{f \in \F \cap r S(L_2)} \frac{1}{N}\sum_{i=1}^N \IND_{\{|f-\pi f| > \delta_f\}}(X_i)~.
\]
The proof
is based on Talagrand's concentration inequality
for the supremum of empirical processes
 \cite{Tal96c} which implies that for $x>0$, with probability at least $1-\exp(-x)$,
$$
Z \leq \frac{3}{2}\E Z + 4\sigma_\F \sqrt{\frac{x}{N}} + \frac{6x}{N}
$$
(see, e.g., Theorems 11.8 and 12.2 in \cite{BoLuMa13}) where
$\sigma_\F = \sup_{f \in \F \cap r S(L_2)} \PROB^{1/2}\{|f-\pi f| > \delta_f\}$.

Here, recalling that $\delta_f  = \Delta \|f\|_{L_2}$ and the fact
that $\delta_f$ is the same for every $f \in \F \cap rS(L_2)$,
by Chebyshev's inequality,
\[
\sigma_\F \leq \sup_{f \in \F \cap r S(L_2)} \frac{\|f-\pi
  f\|_{L_2}}{\delta_f} \leq \frac{\Delta^{3/2}r}{\Delta r} = \frac{\sqrt{\Delta}}{10}~.
\]
Next, let
\begin{equation*}
\phi_\delta(t) =
\begin{cases}
1 &\mbox{if } t \geq \delta~,
\\
\frac {2}{\delta}\left(t-\frac{\delta}{2}\right) & \mbox{if }  t \in \left[\frac{\delta}{2},\delta\right]~,
\\
0 & \mbox{otherwise.}
\end{cases}
\end{equation*}
Therefore, $\phi_\delta(t) \ge \IND_{t>\delta}$ is a Lipschitz function with Lipschitz
constant $2/\delta$ that satisfies $\phi_\delta(0)=0$.
By the Gin\'{e}-Zinn symmetrization theorem \cite{GiZi84} followed by
the contraction inequality for Bernoulli processes \cite{LeTa91},

\begin{eqnarray*}
\E Z & \leq & \E \sup_{f \in \F \cap rS(L_2)} \frac{1}{N} \sum_{i=1}^N \phi_\delta(|f - \pi f|(X_i))
\\
& \leq & 2\E \sup_{f \in \F \cap rS(L_2)} \left|\frac{1}{N} \sum_{i=1}^N \eps_i \phi_\delta(|f - \pi f|(X_i))\right| + \sup_{f \in \F \cap rS(L_2)} \E \phi_\delta(|f - \pi f|(X_i))
\\
& \leq & \frac{4}{\delta} \E \sup_{f \in \F \cap rS(L_2)} \left|\frac{1}{N} \sum_{i=1}^N \eps_i (f - \pi f)(X_i)\right| + \frac{\Delta}{25}~,
\end{eqnarray*}
where the last inequality holds because
\[
\E \phi_\delta(|f - \pi f|(X_i)) \leq \PROB(|f-\pi f|(X) \geq
\delta_f/2) \leq \frac{4 \|f-\pi f\|_{L_2}^2}{\delta_f^2} \leq
\frac{\Delta}{25}~.
\]
Using \eqref{eq:rademacher}, we have
$$
\frac{4}{\delta_f} \E \sup_{f \in \F \cap rS(L_2)} \left|\frac{1}{N} \sum_{i=1}^N \eps_i (f - \pi f)(X_i)\right| \leq \frac{\Delta}{100}~,
$$
and therefore, with probability at least $1-2\exp(-x)$,
\begin{equation*}
\sup_{f \in \F \cap r S(L_2)} \PROB_N(|f-\pi f| > \delta_f)
\leq \frac{3\Delta}{40} + 4\Delta^{1/2} \sqrt{\frac{x}{N}}+ \frac{6x}{N}~.
\end{equation*}
The claim follows by setting $x =c_0 N\Delta$ for a sufficiently small
value of $c_0$.
\endproof

In order to complete the proof of Theorem
\ref{thm:uniform-properties}, it only remains to prove part (b).
The proof is completely analogous with part (a). First
we consider a net of $\F \cap r S(L_2)$ and use part (b') of Theorem
\ref{thm:Properties-via-VC}. Then one may extend the inequality
to all $\F \cap r S(L_2)$ in a similar fashion. In order to avoid
repetition of the same ideas, we omit the details.

\section{Covariance estimation with trimmed means}
\label{sec:variance}

In this section we present the first main component of the mean
estimation procedure. As explained in the introduction, in this first
step we need to estimate the directional variances
$\sigma^2(u) = \var(\inr{X,u})$ in all directions where
$\sigma^2(u)$ is ``not too small.'' This estimator does not need to be
very accurate. It is sufficient for our purposes that the estimator is
correct up to a constant factor.

For this purpose, we require a bit more than the existence of the
covariance matrix $\Sigma$. The key assumption we use is
``$L_q$-$L_2$ norm equivalence'' for some $q>2$. More precisely,
we assume that there exist $q>2$ and $\kappa>0$ such that, for all $u\in S^{d-1}$,
\[
     \left( \EXP \inr{X-\mu, u}^q \right)^{1/q}  \le \kappa      \left( \EXP \inr{X-\mu, u}^2 \right)^{1/2}~.
\]
In other words, writing $\ol{X}=X-\E X$, we assume that for all $u\in S^{d-1}$,
\[
\left\|\inr{\ol{X},u}\right\|_{L_q} \leq  \kappa \left\| \inr{\ol{X},u}\right\|_{L_2}~.
\]
The proposed estimator is quite natural. For each direction $u \in S^{d-1}$,
we compute an appropriately trimmed empirical variance. Unlike standard trimmed mean estimators, where the truncation occurs at a pre-set level, here we trim by removing a fixed number of the largest and smallest values of $\inr{X_i,u}$ corresponding to each direction $u \in S^{d-1}$. To show that
this estimator satisfies the desired properties simultaneously for all
directions, we make use of the tools developed in Sections
\ref{sec:tailintegration} and \ref{sec:uniform}.
The main ingredient of the analysis is Theorem
\ref{thm:uniform-properties}. This theorem requires
that the ``small-ball'' condition \eqref{eq:SB-used} is satisfied. To guarantee this property, we form blocks of a fixed size of the given sample, and take the empirical average within each block. In Section \ref{sec:normequivalence} we show that under $L_q$-$L_2$ norm equivalence, it suffices to form blocks of
constant size (depending in $\kappa$ and $q$).

Let us describe the covariance estimation procedure. In order to estimate variances without knowing the expected values, we
use the standard trick that, if $X'$ is an independent copy of $X$,
then $\var(\inr{X,u}) = (1/2) \EXP \inr{X-X',u}^2$.
It is easy to see that if $X$ satisfies $L_q-L_2$ norm equivalence with constant $\kappa$ then 
$\wt{X}=X-X'$ satisfies $L_q$-$L_2$ norm equivalence with
constant $\sqrt{2}\kappa$.

Thus, we
split the data in two halves to form independent pairs of
observations. For the sake of simpler notation, assume that
we are given $2N$ independent copies, $X_1,\ldots,X_{2N}$
and for $i \in [N]$, define $\wt{X}_i = X_i - X_{N+i}$.

The proposed covariance estimator has two tuning parameters,
$\gamma,\theta \in (0,1)$.

For a positive integer $m$, define
  $Z=\frac{1}{\sqrt{m}} \sum_{i=1}^m \wt{X}_i$.
By Lemma \ref{lem:normequivalence}, there is a constant
$c_1(\kappa,q)$ such that if
$m =\left\lceil \frac{c_1(\kappa,q)}{\gamma^2} \right\rceil$,
then for any $u \in  S^{d-1}$, and for all intervals $I\subset \R$,
\[
\PROB\{ \inr{u,Z} \in I\} \leq \max \left\{L \frac{|I|}{\sigma(u)}, \gamma\right\}~,
\]
where $L > 0$ is a numerical constant.
This implies that every function in the class $\F=\left\{\inr{u,Z}: u  \in B_2^d \right\}$
satisfies the key ``small-ball'' assumption in Theorem \ref{thm:uniform-properties}.

Once the value of $m$ is set, the sample $\wt{X}_1,\ldots,\wt{X}_N$ is
divided into $n=N/m$ blocks, each one of cardinality $m$.
(We may assume, without loss of generality, that $m$ divides $N$.)
For each block $j\in [n]$, we may compute
\[
     Z_j =\frac{1}{\sqrt{m}} \sum_{i=1}^m \wt{X}_{m(j-1)+ i}~.
\]
For every $u \in S^{d-1}$, denote by $J_+(u)$ the set of indices of
the $\theta n$ largest values of $\inr{Z_j,u}$ and
define
\[
\psi_N (u) = \frac{1}{2n} \sum_{j \in [n]\setminus J_+(u)}
\left(\inr{Z_j,u}\right)^2~.
\]

The following theorem summarizes the main performance guarantees of
the estimator $\psi_N(u)$. It is a crucial ingredient of the
mean estimator introduced in the next section.
\begin{proposition} \label{thm:iso-variance}
Assume the condition of Theorem \ref{thm:meanest}.
There are constants $\gamma, \theta \in (0,1)$ and $c_0,c'>0$ depending on $\kappa$
and $q$ for which the following
holds.
Set $m \ge \frac{c_1(\kappa,q)}{\gamma^2}$ and
\[
r^2 = \frac{c_0}{n} \sum_{i \geq c_0 N} \lambda_i~.
\]
Then, with probability at least $1-2\exp(-c' N)$, the estimator $\psi_N$ satisfies
\begin{description}
\item{$(i)$} If $u \in S^{d-1}$ is such that $\sigma(u) \geq r$, then
\[
\frac{1}{4}\sigma ^2(u) \leq \psi_N(u) \leq   2 \sigma ^2(u)~.
\]
\item{$(ii)$} If $\sigma(u) \leq r$ then $\psi_N(u) \leq Cr^2$ for an absolute constant $C$.
\end{description}
\end{proposition}

\begin{proof}
Fix $\gamma,\theta \in (0,1)$ and consider the resulting estimator
$\psi_N$. (Recall that in the definition of $\psi_N$,
the block size $m$ depends on $\gamma$, as well as on the constants
$\kappa$ and $q$ of the norm equivalence condition). We show that
$\gamma$ and $\theta$ may be chosen so that the inequalities of the
theorem hold. In particular,
let $c, c_1$ be the constants appearing in Theorem \ref{thm:uniform-properties}.
We show that it is sufficient if the parameters $\gamma,\theta \in (0,1)$
satisfy $2\theta +8\gamma/c_1 < (5\kappa^2)^{-q/(q-2)}$ and $\theta > 7\gamma/c_1$.

The proof consists of three parts. First we show that $(i)$ holds
above such a the critical level $r$. This is based on Theorems \ref{thm:main-integral}
and \ref{thm:uniform-properties}. Second, we show that
the announced value of $r$ satisfies $r > \ol{\rho}_n(c,\Delta,L)$ and
therefore it is a valid choice.
Finally,  we prove part $(ii)$ of the
theorem, based on the \emph{small-ball method}.
Before the proof, we establish some consequences of the
$L_q-L_2$ norm equivalence condition \eqref{eq:normequivalence}.


\subsection{$L_q-L_2$ norm equivalence implies a small ball property}
\label{sec:normequivalence}

Let $Y$ be an absolutely continuous real-valued random variable that
satisfies $\|Y-\E Y\|_{L_q} \leq \kappa \|Y-\E Y\|_{L_2}$ for some
$q>2$ and $\kappa >0$.
Let $Y_1,\ldots,Y_m$ be independent copies of $Y$.
To ease notation, set $\ol{Y}=Y-\E Y$.

\begin{lemma}
\label{lem:normequivalence}
Define $Z_m=\frac{1}{\sqrt{m}} \sum_{i=1}^m \ol{Y}_i$.
\begin{itemize}
\item
There exists a constant $m_0(q,\kappa)$ such that
 if $m \geq m_0(q,\kappa)$, then
\begin{equation} \label{eq:fact1}
\PROB\left\{ Z_m \geq 0 \right\} \geq \frac{1}{4}
\quad \text{and} \quad
\PROB\left\{  Z_m \leq 0 \right\} \geq \frac{1}{4}~.
\end{equation}
\item
Let $0<\xi<1/2$.  Then
\begin{equation} \label{eq:fact2}
\E \ol{Y}^2 \IND_{\{ |\ol{Y}| \geq Q_{1-\alpha}(|\ol{Y}|)\}} \leq \xi \E \ol{Y}^2~,
\end{equation}
where $\alpha=(\xi/\kappa^2)^{q/(q-2)}$.
\item
\begin{equation} \label{eq:fact3}
\|Z_m\|_{L_q} \leq (4(q-1))^{1/2} \kappa \|Z_m\|_{L_2}~.
\end{equation}
\item
There exists a numerical constant $L$ and
a constant $c_1=c_1(\kappa,q)$ such that,
for any $\gamma \in (0,1)$,  if $m \ge c_1/\gamma^2$, then for all
intervals $I\subset \R$,
\begin{equation} \label{eq:fact4}
\PROB\{ Z_m \in I\} \leq \max \left\{L \frac{|I|}{\left\|\ol{Y}\right\|_{L_2}}, \gamma\right\}~.
\end{equation}
\end{itemize}
\end{lemma}

\begin{proof}
\eqref{eq:fact1} follows from a generalization of the Berry-Esseen theorem (see, e.g., \cite{MoODOl10}),

To prove \eqref{eq:fact2}, note that by H\"older's inequality for $\beta=q/2$ and the $L_q-L_2$ norm equivalence,
\[
\E \left[\ol{Y}^2 \IND_{\{ |\ol{Y}| \geq Q_{1-\alpha}(\ol{Y})\}} \right]
\leq \left\|\ol{Y}\right\|_{L_q}^2 \PROB\left\{ \ol{Y} \geq Q_{1-\alpha}(|\ol{Y})| \right\}^{1-\frac{2}{q}} \leq \kappa^2 \E \ol{Y}^2 \cdot \alpha^{1-\frac{2}{q}}~.
\]
For the proof of \eqref{eq:fact3}, observe that, if $\eps_1,\ldots,\eps_m$ are independent symmetric
Bernoulli random variables with $\PROB\{\epsilon_i=1\}
=\PROB\{\epsilon_i=- 1\} =1/2$, then by symmetrization and Khintchine's
inequality (see, e.g, \cite[p.21]{DeGi99},
\begin{equation*}
\E |Z_m|^q \leq 2^q \E \left|\frac{1}{\sqrt{m}}\sum_{i=1}^m \eps_i \ol{Y}_i \right|^q \leq (4(q-1))^{q/2}\E  \left|\frac{1}{m} \sum_{i=1}^m \ol{Y}_i^2 \right|^{\frac{q}{2}} \leq (4(q-1)\kappa^2)^{q/2}\E \ol{Y}^2~,
\end{equation*}
where the last inequality follows from the norm-equivalence condition \eqref{eq:normequivalence} because, by the convexity of $\phi(t)=|t|^{q/2}$,
\[
\left|\frac{1}{m} \sum_{i=1}^m \ol{Y}_i^2 \right|^{\frac{q}{2}} \leq \frac{1}{m} \sum_{i=1}^m \left|\ol{Y}_i\right|^q~.
\]
It remains to prove the small-ball bound of \eqref{eq:fact4}.
Setting $\xi=1/50$, it follows from \eqref{eq:fact2} that
$\E \left[\ol{Y}^2 \IND_{\{ |\ol{Y}| \geq Q_{1-\alpha}(|\ol{Y}|)\}}
\right] \leq
(1/50) \E \ol{Y}^2$ where $\alpha=(1/(50\kappa^2)^{q/(q-2)}$.
Moreover, by Chebychev's inequality, $Q_{1-\alpha}(|\ol{Y}|) \leq
\left\|\ol{Y}\right\|_{L_2}/\sqrt{\alpha}$, and therefore
\[
\E \left[ \ol{Y}^2 \IND_{\{ |\ol{Y}| \geq
    \left\|\ol{Y}\right\|_{L_2}/\sqrt{\alpha} \}} \right] \leq \frac{\E\left[ \ol{Y}^2\right]}{50}~.
\]
By the second part of Theorem 3.2 in Mendelson \cite{Men20a}, there exists a
constant $c_1$ depending on $\alpha$ (and hence on $\kappa$ and $q$) such that,
for any $\gamma \in (0,1)$,  if $m \ge c_1/\gamma^2$, then
\[
\sup_{x \in R} \PROB\{ |Z_m-x| \leq c_2 \gamma \left\|\ol{Y}\right\|_{L_2}\} \leq \gamma~,
\]
where $c_2$ is an absolute constant. Hence, for any interval $I \subset \R$,
\[
\PROB\{ Z_m \in I \} \leq \max \left\{\frac{|I|}{c_2
    \left\|\ol{Y}\right\|_{L_2}}, \gamma \right\}~.
\]
\end{proof}

\subsubsection*{Above the critical level}

We start by proving part $(i)$ of Proposition
\ref{thm:iso-variance}.
The proof is based on applying
Theorem \ref{thm:uniform-properties} for the class
$\F=\left\{\inr{u,\cdot}: u  \in B_2^d \right\}$.

Let $\Delta = \gamma/c_1$ and let $\ol{\rho}_n(c,\Delta,L)$ be the
critical level of the class $\F$, as defined in Definition \ref{def:criticallevel}.
Let $r > \ol{\rho}_n(c,\Delta,L)$ be arbitrary.

Let $D=\{u\in \Rd: \sigma(u) \le 1\}$.
Clearly, $\sqrt{2}\sigma(u)=\|\inr{u,Z}\|_{L_2}$ and for each $u$,
$\inr{u,Z}$ is a symmetric random variable.

The class $\F=\left\{\inr{u,\cdot}: u \in B_2^d\right\}$ is star-shaped
around $0$ and therefore it satisfies the conditions of Theorem
\ref{thm:uniform-properties}.
By Theorem \ref{thm:uniform-properties}, there is a numerical constant
$c_2>0$ and an event ${\cal A}$
of probability at least $1-2\exp(-c_2 \gamma n)$, on which the following
holds: for every $u \in B_2^d$ such that $\left\|\inr{u,\ol{X}}\right\|_{L_2} \geq r$,
the random variable $\inr{u,Z}$ satisfies properties $(1)$-$(2)$ of Definition
\ref{def:ratioconditions}.
Moreover, property $(3)$ holds
trivially for every $\inr{u,Z}$ by the symmetry of these random
variables.

Suppose that the event ${\cal A}$ occurs and let $u \in B_2^d$ be such that $\sigma(u) \geq r$.

Let $\left(\inr{u,Z_j}^\sharp\right)_{j=1}^n$ be the monotone nonincreasing
rearrangement of $\inr{u,Z_1},\ldots, \inr{u,Z_n}$.
Define $\wh{Q}(u) = \inr{u,Z_{\theta n}}^\sharp$ and denote by $Q_q(u)$
the $q$-quantile of the random variable $\inr{u,Z}$. By Lemma \ref{lemma:est-on-hat-Q} and the symmetry of the random variables, it follows that
\[
Q_{1-(2\theta+8\Delta)}(u) \leq \wh{Q}(u) \leq Q_{1-(2\theta-8\Delta)/3}(u)~.
\]
Setting
\[
Q_3 \defeq Q_{1-(2\theta+8\Delta)}(u) \quad \text{and} \quad Q_4 \defeq Q_{1-(2\theta-8\Delta)/3}(u)~,
\]
just as in the proof of Lemma \ref{lemma:int-preliminary},
we have
\[
2 \psi_N(u) \geq \int_0^{Q_3} 2t\PROB_n\left\{ |\inr{u,Z}|>t
\right\} dt - \theta \wh{Q}^2(u)
\]
and
\[
2 \psi_N(u) \leq  \int_0^{Q_4} 2t\PROB_n\left\{|\inr{u,Z}|>t \right\} dt~.
\]
Since $\theta > 7\Delta$, we have $(2\theta-8\Delta)/3 \ge \Delta$ and
therefore if $0 \leq t \leq Q_4$ then $\PROB\left\{ |\inr{u,Z}| > t
\right\} \geq \Delta$.
Thus, by Theorem \ref{thm:uniform-properties}, for  $t\in [0, Q_4]$,
\[
\frac{1}{2} \PROB\left\{ |\inr{u,Z}| > t \right\} \leq \PROB_n\left\{
  |\inr{u,Z}| > t \right\} \leq \frac{3}{2}\PROB\left\{
  |\inr{u,Z}| > t \right\}~.
\]
Also, just as in \eqref{eq:hat-Q-tail},
\[
\theta \wh{Q}^2  \le 2 \E \left[ \inr{u,Z}^2\IND_{\{|\inr{u,Z}|
    \geq Q_3\}} \right]~.
\]
Hence,
\[
2 \psi_N(u) \geq \frac{1}{2} \E \inr{u,Z}^2 - \frac{5}{2} \E \left[
  \inr{u,Z}^2\IND_{\{|\inr{u,Z}| \geq Q_3\}} \right] \geq \frac{1}{4} \E \inr{u,Z}^2,
\]
provided that $5\E \inr{u,Z}^2\IND_{\{|\inr{u,Z}| \geq Q_3\}} \leq \E
\inr{u,Z}^2$, which holds by \eqref{eq:fact2} of Lemma
\ref{lem:normequivalence} whenever $2\theta+ 8\Delta \le (5\kappa^2)^{-q/(q-2)}$.

In the reverse direction,
\[
2 \psi_N(u) \leq  \int_0^{Q_4} 2t\PROB_n\left\{ |\inr{u,Z}|>t
\right\} dt \leq 2\int_0^\infty 2t\PROB\left\{ |\inr{u,Z}|>t \right\} dt = 2 \E \inr{Z,u}^2~,
\]
and combining the two inequalities we have that for every $v \in \R^d$,
\[
\frac{1}{4} \sigma ^2(u) \leq \psi(v) \leq 2\sigma^2(u)^2~,
\]
as claimed.

\subsubsection*{The critical level}

Next we show that there exists a constant $c_0$ such that, for the
class of functions $\F=\left\{\inr{u,\cdot}: u \in B_2^d\right\}$,
\[
r = \sqrt{\frac{c_0}{n} \sum_{i \geq c_0 n} \lambda_i} \ge  \ol{\rho}_n(c,\Delta,L)
\]
for the values of $c,\Delta,L$ introduced the first part of the proof
above. (Recall Definition \ref{def:criticallevel} where
$\ol{\rho}_n(c,\Delta,L)$ was introduced; also note that
$c$ is an absolute constant while $\Delta$ and $L$ depend on
the constants $\kappa$ and $q$ of the norm-equivalence condition \eqref{eq:normequivalence}).

Without loss of generality, we may assume that the covariance matrix
$\Sigma$ is positive definite.
We may write the random vector $Z=X-X'$ as $Z=TW$ where the random
vector $W$ has identity covariance matrix and $T:\Rd\to\Rd$ is a positive
definite linear transformation.
Since $Z/\sqrt{2}$ has the same
covariance matrix $\Sigma$ as $X$, the eigenvalues of $T$ are $2\lambda_1, \ldots,2\lambda_d$.

It is straightforward to verify that $D=\{u\in \Rd: \sigma(u) \le 1\}
=T^{-1} B_2^d$ and therefore the packing numbers of Definition \ref{def:criticallevel} satisfy
\[
{\cal M}(B_2^d \cap r D, \Delta^{3/2} r /(1000 \cdot L))
={\cal M}\left(TB_2^d \cap r B_2^d, \Delta^{3/2} r /(1000\cdot L) \right)~.
\]
Also, $TB_2^d \cap r B_2^d \subset {\cal E}$ for an ellipsoid ${\cal
  E}$ whose principal axes are of lengths proportional to
the values $\min\{\sqrt{\lambda_i}, r\}$. By Sudakov's inequality (see, e.g. \cite{LeTa91}), there is a
constant $c_3$ (depending on $L$ only) such that
\[
c_3 \Delta^{3/2} r \log^{\frac{1}{2}} {\cal M} \left(TB_2^d \cap r
  B_2^d, \Delta^{3/2}/ (1000\cdot L) \right) \leq \E \sup_{x \in {\cal E}} \inr{G,x}~,
\]
where $G$ is the standard Gaussian random vector in $\R^d$. A straightforward computation shows that
\[
\E \sup_{x \in {\cal E}} \inr{G,x} \leq c_4\left(\sum_{i=1}^d \min\left\{\lambda_i,r^2\right\}\right)^{\frac{1}{2}}
\]
for a numerical constant $c_4>0$
and, in particular, inequality (1) in Definition \ref{def:criticallevel} holds if
\begin{equation} \label{eq:cond-on-r-3}
\left(\sum_{i=1}^d \min\left\{\lambda_i,r^2\right\}\right)^{\frac{1}{2}} \leq c_5 \Delta^2 \sqrt{n} r~,
\end{equation}
where the constant $c_5$ depends on $L$.

Turning to inequality (2) in Definition \ref{def:criticallevel}, observe that
\begin{eqnarray*}
\E \sup_{u \in 2B_2^d \cap \Delta^{3/2}r D} \left|\sum_{i=1}^n \eps_i
  \inr{Z_i,u}\right|
& \leq & 2 \E \sup_{u \in T^{-1} (T B_2^d \cap \Delta^{3/2}r B_2^d)} \left|\sum_{i=1}^n \eps_i \inr{W_i,T u}\right|
\\
& = & 2 \E \sup_{v \in T B_2^d \cap \Delta^{3/2}r B_2^d}
      \left|\sum_{i=1}^n \eps_i \inr{W_i,v}\right| \\
& \le & c_6 \E \sup_{v \in {\cal E}'} \left|\sum_{i=1}^n \eps_i \inr{W_i,v}\right|
\end{eqnarray*}
for an ellipsoid ${\cal E}'$ that may be written as ${\cal E}'= Q
B_2^d$ for a linear transformation $Q$ whose eigenvalues are
proportional to $\min\{\sqrt{\lambda_i}, \Delta^{3/2} r\}$. Thus,
\[
\E \sup_{v \in {\cal E}'} \left|\sum_{i=1}^n \eps_i \inr{W_i,v}\right| = \E \sup_{v \in B_2^d} \left|\sum_{i=1}^n \eps_i \inr{QW_i,v}\right| \leq \sqrt{n} \left(\E \|QW\|_2^2\right)^{1/2}~.
\]
Since $W$ is isotropic,
\[
\E \|Q W\|_2^2 = \sum_{i=1}^d \E \inr{W,Q^* e_i}^2 = \sum_{i=1}^d
\|Q^*e_i\|_2^2 \le c_7 \sum_{i=1}^d \min\{\lambda_i, \Delta^{3} r^2\}~.
\]
Therefore, inequality (2) in Definition \ref{def:criticallevel} is verified once
\begin{equation} \label{eq:cond-on-r-4}
\left(\sum_{i=1}^d \min\{\lambda_i, \Delta^{3} r^2\}\right)^{\frac{1}{2}} \leq c_8 \Delta^2 \sqrt{n} r
\end{equation}
for a constant $c_8$ (that depends on $\kappa$ and $q$).
Recalling that $\Delta <1$, it is evident that both \eqref{eq:cond-on-r-3} and \eqref{eq:cond-on-r-4}
are satisfied when
\begin{equation} \label{eq:cond-on-r-5}
\left(\sum_{i=1}^d \min\{\lambda_i, r^2\}\right)^{\frac{1}{2}} \leq c_9 \sqrt{n} r
\end{equation}
for a constant $c_9=c_9(\kappa,q)$.
Using that $\min\{\lambda_i, r^2\} \leq r^2$ for $i \leq \frac{1}{2}c_9$, it suffices that
\[
r^2 \geq c_0 \frac{1}{n} \sum_{i \geq c_0 n} \lambda_i~,
\]
for a constant depending on $\kappa$ and $q$ only, as claimed.

\subsubsection*{Below the critical level}

Finally, we prove part $(ii)$ of Proposition \ref{thm:iso-variance}.
Let the parameters $\theta,\gamma \in (0,1)$ be constant as specified
in the proof of part $(i)$ above.

Fix $r_0>0$ that satisfies \eqref{eq:cond-on-r-5} and set
$\F_{r_0} = \{ \inr{u,\cdot} : u \in B_2^d \cap r_0 D\}$.
Our goal is to show that, with high probability, for every $u \in
B_2^d \cap r_0 D$,
\[
\left(\inr{u,Z}_{\theta n})\right)^{\sharp} \leq c_0 \frac{r_0}{\sqrt{\theta}}
\]
for a constant $c_0$,
where $\left(\inr{u,Z}_j)^\sharp\right)_{j=1}^n$ is the monotone nonincreasing
rearrangement of $\inr{u,Z_1},\ldots, \inr{u,Z_n}$.
On this event, for all $u \in B_2^d \cap r_0 D$, we have
\[
\psi_N(u) \leq \frac{c_0^2}{\theta} r_0^2~,
\]
as required.

The proof uses a net argument.
Let $U_{r_0}$ be a maximal $\sqrt{\theta}r_0$-separated subset of
$B_2^d \cap r_0 D$. A standard argument using the norm equivalence
condition, the union bound, and Markov's inequality shows that, for every $u \in U_{r_0}$, with probability at least $1-\exp(-c_1 k \log(n/(c_2k)))$,
\[
\left(\inr{u,Z}\right)^{\sharp}_{k}\leq  \|\inr{u,Z}\|_{L_2} \sqrt{\frac{n}{k}}~,
\]
where $c_1=q/2-1$ and $c_2=(e\kappa^q)^{2/(q-2)}$.
Hence, by setting $k=\theta n/2$ and if
\[
|U_{r_0}| \leq (c_1/4) n \theta \log\left(\frac{2c_2}{\theta}\right)~,
\]
that is, if
\begin{equation} \label{eq:cond-on-rho-2}
\log{\cal M}(B_2^d \cap r_0 D, \sqrt{\theta} r_0) \leq (c_1/4) n \theta \log\left(\frac{2c_2}{\theta}\right)~,
\end{equation}
it follows that, with probability at least $1- \exp\left(-(c_1/4) n \theta \log\left(\frac{2c_2}{\theta}\right)\right)$,
for every $u \in U_{r_0}$,
\[
\left(\inr{u,Z}\right)^{\sharp}_{\theta n/2} \leq  \frac{r_0}{\sqrt{\theta/2}}~.
\]
Denote by $\pi u$ the best approximation to $u$ in $U_{r_0}$ with
respect to the $L_2(Z)$ norm.
In particular, by the choice of $U_{r_0}$, $\|u-\pi u\|_{L_2} \leq
\sqrt{\theta}r_0$. To complete the proof it suffices to show that,
with high probability,
\[
\Gamma \defeq \sup_{v \in B_2^d \cap r_0 D} \left| \left\{ j : \left|
      \inr{u-\pi u,Z_j} \right| \geq  \frac{8 r_0}{\sqrt{\theta}} \right\} \right| \leq \frac{\theta n}{2}~.
\]
 This follows from what is, by now, a standard
argument:

By the bounded differences inequality we have that, with probability at least $1-2\exp(- \theta^2 n/8)$, $\Gamma \leq \frac{\theta n}{2}$, provided that
\[
\E \Gamma \leq \frac{\theta n}{4}~.
\]
By symmetrization and contraction,
\begin{eqnarray*}
\E \Gamma & \leq & \frac{\sqrt{\theta}}{8 r_0} \E \sup_{u \in B_2^d \cap r_0 D} \sum_{j=1}^n |\inr{u-\pi u,Z_i}|
\\
& \leq & \frac{\sqrt{\theta}}{8 r_0} \left( \E \sup_{u \in B_2^d \cap r_0 D} \sum_{j=1}^n \left(|\inr{u-\pi u,Z_i}|-\E|\inr{u-\pi u,Z_i}| \right) + \sqrt{\theta} r_0 n \right)
\\
& \leq & \frac{\sqrt{\theta}}{8 r_0} \left(2 \E \sup_{u \in B_2^d \cap r_0 D} \left| \sum_{j=1}^n \eps_j \inr{u-\pi u,Z_i}\right| + \sqrt{\theta} r_0 n\right)
\\
& \leq & \frac{\sqrt{\theta}}{8 r_0} \left(4\E \sup_{u \in B_2^d \cap \sqrt{\theta}r_0 D}\left| \sum_{j=1}^n \eps_j \inr{u,Z_i}\right| + \sqrt{\theta} r_0 n\right)~.
\end{eqnarray*}
Hence, $\E \Gamma \le \theta n/4$ provided that
\[
\E \sup_{u \in B_2^d \cap \sqrt{\theta}r_0 D}\left| \sum_{j=1}^n \eps_j \inr{u,Z_i}\right|
 \leq \frac{r_0\sqrt{\theta} n}{4}~.
\]
This may be proved by the same argument used in the second part of the
proof above. In particular, the inequality holds once $r_0 \ge
\sqrt{(c_0/n)\sum_{i\ge c_0 n} \lambda_i}$ for an appropriate
constant (depending on $\kappa$ and $q$), as required.
\end{proof}

\section{Multivariate mean estimator and its performance}
\label{sec:meanest}

Now we are prepared to define the mean estimator announced in Theorem
\ref{thm:meanest} and prove its performance bound. The estimator
receives, as input, $3N$ independent, identically distributed random
vectors $X_1,\ldots,X_{3N}$, the parameters $\kappa>0$ and $q>2$, and
the confidence parameter $\delta \in (0,1)$. (The sample size is set
to be $3N$ for convenience as the proposed estimator splits the data
into three equal parts.)

The data $X_{N+1},\ldots,X_{3N}$ are used to estimate the variances
$\sigma^2(u)= \var\left(\inr{X,u}\right)$ for $u\in S^{d-1}$. Using
the estimator $\psi_N$, we have that, on an event
${\cal A}$ of probability at least $1-e^{-cN}$,
\begin{equation}
\label{eq:goodevent}
\begin{array}{ll}
\frac{1}{4}\sigma^2(u) \leq \psi_N(u) \leq   2 \sigma ^2(u) &
                                                              \text{for all $u \in S^{d-1}$ such that $\sigma(u) \geq r$} \\
\psi_N(u) \leq Cr^2 & \text{otherwise}
\end{array}
\end{equation}
Here $c,C$ are constants
depending on $\kappa$ and $q$ only and
\[
    r = \sqrt{\frac{c_0}{N}\sum_{i\ge c_0N} \lambda_i}
\]
for another constant $c_0>0$ depending on $\kappa$ and $q$.
(Recall that the variance estimator $\psi_N$ has two parameters
$\theta$ and $\gamma$, both may be determined by $\kappa$ and $q$.)

The data $X_1,\ldots,X_N$ are used to estimate the mean $\EXP\inr{X,u}
= \inr{\mu,u}$ for all $u\in S^{d-1}$. One would like to construct
an estimator such that it is approximately correct
simultaneously for all directions $u\in S^{d-1}$. To this end,
similarly to the covariance estimation procedure of the previous
section, we  divide the sample $X_1,\ldots,X_N$ into $n = N/m$ blocks,
each of size $m$, and compute, for $j\in [n]$,
\[
    Y_j = \frac{1}{\sqrt{m}}\sum_{i=1}^{m} X_{m(j-1) + i}~.
\]
The recommended value of $m$ is specified below. It is not necessarily
the same as the block size in the covariance estimation
procedure defined in Section \ref{sec:variance}.
However, the role of this blocking procedure is the
same as in the covariance estimation procedure of
Section \ref{sec:variance}: by an appropriate choice of $m$, the random vectors
$Y_j$ satisfy the small-ball condition that allows us to apply
Theorem \ref{thm:uniform-properties}. The estimator is a simple
trimmed-mean estimator defined as
\[
    \wh{\nu}_N(u)= \frac{1}{\sqrt{m}} \frac{1}{n-2\theta n} \sum_{j\in [n] \setminus
      J_+(u) \cup J_-(u)} Y_j~,
\]
where $\theta \in (0,1/2)$ is a parameter of the estimator and the sets
$J_+(u)$ and $J_-(u)$ correspond to the indices of the $\theta n$
smallest and $\theta n$ largest values of $\inr{Y_j,u}$. (We may
assume that $\theta n$ is an integer and the value of $\theta$ is
specified below.) The key property of the marginal mean estimator
$\wh{\nu}_N(u)$ is summarized in the next proposition.

\begin{proposition}
\label{prop:marginal}
Assume the condition of Theorem \ref{thm:meanest}.
There exist choices of the parameters $m$ and $\theta$ of the
estimator $\wh{\nu}_N(u)$ that depend only on $\kappa$ and $q$ and
there exist constants $c,C'>0$ depending on $\kappa,q$ such that,
with probability at least $1-\delta$, for all $u\in S^{d-1}$,
\begin{equation}
\label{eq:goodeventmarginal}
      \left| \wh{\nu}_N(u)- \inr{\mu,u} \right| \le C'\left( \sqrt{\frac{\sigma^2(u)\log(1/\delta) }{N}} + \sqrt{\frac{\sum_{i=c\log(1/\delta)}^d \lambda_i}{N}} \right)~.
\end{equation}
\end{proposition}

The proof of the proposition follows from Theorems
\ref{thm:main-integral} and \ref{thm:uniform-properties}, similarly to
the arguments presented in the previous section for covariance
estimation. In order to avoid repetitions, we defer the details to
Section \ref{sec:proofmarginal} in the Appendix.

Equipped with Proposition \ref{prop:marginal}, it is now easy to
define the mean estimator announced in the introduction and prove
Theorem \ref{thm:meanest}.

Let $\psi_N(u)$ and $\wh{\nu}_N(u)$ be the variance and marginal
mean estimators defined above.
For a parameter $\rho>0$, and for each $u\in S^{d-1}$, define the
slabs
\[
   E_{u,\rho} = \left\{ v\in \Rd: \left| \wh{\nu}_N(u) - \inr{v,u}
     \right|  \le \rho + 2 C' \sqrt{\frac{\psi_N(u) \log(1/\delta)}{N}} \right\}
\]
and let
\[
S_{\rho} = \bigcap_{u\in S^{d-1}} E_{u,\rho}~.
\]
Note that for every $\rho>0$, the set $S_{\rho}$ is a compact set and
that $S_{\rho}$ is nonempty for a sufficiently large
$\rho$. Therefore, the set
\[
    S = \bigcap_{\rho>0: S_{\rho} \neq \emptyset} S_{\rho}
\]
is not empty. We define the mean estimator as any element $\mu_N \in S$.

Theorem \ref{thm:meanest} now follows easily.

\medskip
\noindent
{\bf Proof of Theorem \ref{thm:meanest}.}
By Propositions \ref{thm:iso-variance} and \ref{prop:marginal}, with
probability at least $1-\delta$, both \eqref{eq:goodevent} and
\eqref{eq:goodeventmarginal} hold,
where $\delta \in (0,1)$ is such that $c\log(1/\delta) \le c_0N$.
Denote this event by ${\cal A}$.
(Here $c_0$ is the constant appearing in the definition of $r$ and $c$
is as in \eqref{eq:goodeventmarginal}.)

On the event ${\cal A}$, if $u\in S^{d-1}$ is such that $\sigma(u) \ge
r$, then $\sigma^2(u) \le 4\psi_N(u)$ and therefore,
\eqref{eq:goodeventmarginal} implies that
\[
 \left| \wh{\nu}_N(u)- \inr{\mu,u} \right| \le\rho + 2 C' \sqrt{\frac{\psi_N(u) \log(1/\delta)}{N}}
\]
whenever $\rho \ge C'\sqrt{\frac{\sum_{i=c\log(1/\delta)}^d \lambda_i}{N}}$.
On the other hand, if $\sigma(u) \ge r$, then $\psi_N(u) \le Cr^2$,
and therefore, by bounding the right-hand side of
\eqref{eq:goodeventmarginal} further, we get
\begin{eqnarray*}
   \left| \wh{\nu}_N(u)- \inr{\mu,u} \right|
& \le & C'\left(
     \sqrt{\frac{C r^2\log(1/\delta) }{N}} +
        \sqrt{\frac{\sum_{i=c\log(1/\delta)}^d \lambda_i}{N}}\right)
  \\
& \le & C'\sqrt{\frac{Cc_0}{c}} r +
       C' \sqrt{\frac{\sum_{i=c\log(1/\delta)}^d \lambda_i}{N}}
\\
& & \text{(since $c\log(1/\delta) \le c_0N$)}  \\
& \le & C'\left(c_0\sqrt{\frac{C}{c}} +1\right)\sqrt{\frac{\sum_{i=c\log(1/\delta)}^d \lambda_i}{N}}~,
\end{eqnarray*}
where at the last step again we used the fact that $c\log(1/\delta) \le
c_0N$.

Hence, on the event ${\cal A}$, we have that $\mu \in S_{\rho}$
when
\[
     \rho = C_1 \sqrt{\frac{\sum_{i=c\log(1/\delta)}^d \lambda_i}{N}}~,
\]
where $C_1 = C'\left(c_0\sqrt{\frac{C}{c}} +1\right)$. Thus,
$S_{\rho}$ is nonempty and, by definition, $\wh{\mu}_N \in S_{\rho}$
for this choice of $\rho$. This means that, for all $u\in S_{\rho}$,
\[
   \left| \inr{\wh{\mu}-\mu,u} \right| \le C_1
   \sqrt{\frac{\sum_{i=c\log(1/\delta)}^d \lambda_i}{N}}
 + 2 C' \sqrt{\frac{\psi_N(u) \log(1/\delta)}{N}}~.
\]
The theorem is now proved by noticing that $\psi_N(u) \le 2\sigma^2(u)$ when $\sigma(u) \ge r$ and
$\psi_N(u) \le Cr^2$ otherwise.
\endproof

\appendix
\section{Appendix: additional proofs}

\subsection{Proof of Proposition \ref{prop:empmean}}
\label{sec:proofofempmean}

We may assume, without loss of generality, that $\Sigma$ is diagonal such that the eigenvalues
$\lambda_1 \ge \lambda_2 \ge \cdots \ge \lambda_d$ have the canonical basis vectors $e_1,\ldots,e_d$
as corresponding eigenvectors. Denote by $Y= \sqrt{N} \left(\wt{\mu}_N - \mu\right)$
and note that $Y$ is a zero-mean Gaussian vector with covariance
matrix $\Sigma$.

Actually, we prove the lower bound
\begin{equation}
\label{eq:twoterms}
   S \ge C' \left(\sqrt{\frac{\sum_{i>k_0}\lambda_i}{N}}
 +     \sqrt{\frac{\lambda_{k_0+1}\log(1/\delta)}{N}}\right)~,
\end{equation}
which is seemingly stronger than the announced inequality. However,
note that the second term in the expression of the lower bound of the
strong term above satisfies
\[
   \sqrt{\frac{\lambda_{k_0+1}\log(1/\delta)}{N}} \sim
   \sqrt{\frac{k_0 \lambda_{k_0+1}}{N}}
\le
   \sqrt{\frac{2\sum_{i>k_0/2} \lambda_i}{N}}~.
\]
Hence, we do not lose much by ignoring the second term.

Suppose that (\ref{eq:ass}) holds on an event $\Omega_\delta$ such
that $\PROB\{\Omega_{\delta}\}\ge 1-\delta$. Let $k$ be the unique value such
that $S \in (C\sqrt{\lambda_{k+1}\log(1/\delta)/N}, C\sqrt{\lambda_k\log(1/\delta)/N}]$.

Denote by $U_k\subset \Rd$ the vector space
spanned by $e_1,\ldots,e_k$.
For all $u\in U_k \cap S^{d-1}$ we have $\sigma^2(u) \ge \lambda_k$, and for all
such vectors $S \le C \sigma(u)\sqrt{\log(1/\delta)/N}$. Therefore,
on the event $\Omega_{\delta}$,
\[
\forall u \in U_k \cap S^{d-1}: \ \inr{\wt{\mu}_N - \mu, u} \le  2 C \sqrt{\frac{\sigma^2(u)\log(1/\delta) }{N}}~.
\]
Equivalently,
\[
     \sup_{u\in U_k \cap S^{d-1}} \frac{\inr{Y,u}}{\sigma(u)}\le 2C \log(1/\delta)~.
\]
If $G=(G_1,\ldots,G_d)$ is a standard normal vector in $\Rd$, then we
may write $Y=\Sigma^{1/2}G$. Since $\sigma(u)= \|\Sigma^{1/2}u\|$, and
$\Sigma$ is a diagonal matrix,
\[
   \sup_{u\in U_k \cap S^{d-1}} \frac{\inr{Y,u}}{\sigma(u)}=
   \sup_{v\in U_k \cap S^{d-1}} \inr{G,v} = \|G^{(k)}\|~,
\]
where $G^{(k)}= (G_1,\ldots,G_k)$ is a standard normal vector in
$\R^k$. By the Gaussian concentration inequality, with probability at
least $1-\delta$, we have
\[
   \|G^{(k)}\| \ge \EXP \|G^{(k)}\| - \sqrt{2\log(1/\delta)}\ge
   \sqrt{k-1} - \sqrt{2\log(1/\delta)}~.
\]
Comparing the upper and lower bounds for $\|G^{(k)}\|$, we conclude
that
\[
    k \le 1+ (2C+\sqrt{2})^2 \log(1/\delta) = k_0~,
\]
implying that
\begin{equation}
\label{eq:firstboundforS}
   S \ge C\sqrt{\frac{\lambda_{k_0+1}\log(1/\delta)}{N}}~.
\end{equation}
Next we consider the orthogonal complement $U_k^{\perp}$ of $U_k$.
Since $\sigma^2(u) \le \lambda_{k+1}$ for all $u\in U_k^{\perp} \cap
S^{d-1}$, on the event $\Omega_{\delta}$, we have
\[
\sup_{u \in U_k^{\perp} \cap S^{d-1}} \inr{Y,u}\le  2S\sqrt{N}~.
\]
Writing $Y=\Sigma^{1/2}G$ as before, and noting that
$\sup_{u \in U_k^{\perp} \cap S^{d-1}} \inr{\Sigma^{1/2}G,u}$ is a
Lipschitz function of $G$ with constant $\sqrt{\lambda_{k+1}}$, the
Gaussian concentration inequality implies that, with probability at
least $1-\delta$,
\[
    \sup_{u \in U_k^{\perp} \cap S^{d-1}} \inr{Y,u} \ge \EXP \sup_{u
      \in U_k^{\perp} \cap S^{d-1}} \inr{Y,u}-
    \sqrt{2\lambda_{k+1}\log(1/\delta)}
   = \EXP \|\Sigma^{1/2} G^{(k)\perp}\| -
    \sqrt{2\lambda_{k+1}\log(1/\delta)}~,
\]
where $G^{(k)\perp}= (G_{k+1},\ldots,G_d)$.
By the Gaussian Poincar\'e inequality (see., e.g., \cite[Theorem 3.20]{BoLuMa13}),
\[
\EXP \|\Sigma^{1/2} G^{(k)\perp}\| \ge \sqrt{\EXP \|\Sigma^{1/2}
  G^{(k)\perp}\|^2} - \sqrt{\lambda_{k+1}}
  = \sqrt{\sum_{i>k} \lambda_i} - \sqrt{\lambda_{k+1}}~,
\]
and therefore
\[
   S \ge \frac{1}{2\sqrt{N}} \left( \sqrt{\sum_{i>k} \lambda_i} -
     \sqrt{\lambda_{k+1}} \left(1+\sqrt{2\log(1/\delta)}\right)\right)~.
\]
If
\[
    \sqrt{\lambda_{k+1}} \left(1+\sqrt{2\log(1/\delta)}\right) \le
    \frac{1}{2} \sqrt{\sum_{i>k} \lambda_i}~,
\]
then
\[
   S \ge \frac{1}{4\sqrt{N}} \left( \sqrt{\sum_{i>k} \lambda_i}\right)
\]
which, together with (\ref{eq:firstboundforS}) and the fact that $k\le
k_0$, implies (\ref{eq:twoterms}). On the other hand, if
\[
    \sqrt{\lambda_{k+1}} \left(1+\sqrt{2\log(1/\delta)}\right) >
    \frac{1}{2} \sqrt{\sum_{i>k} \lambda_i}~,
\]
then (\ref{eq:firstboundforS}) already implies inequality (\ref{eq:twoterms}).

\subsection{Proof of Proposition \ref{prop:marginal}}
\label{sec:proofmarginal}

For $j\in [n]$, we write $\ol{Y}_j=Y_j-\E Y_j = Y_j-\mu$.
Then for all $u\in S^{d-1}$,
\[
\left|\wh{\nu}(u) - \inr{\mu,u} \right| = \frac{1}{\sqrt{m}}
\left(\frac{1}{n-2n\theta} \sum_{j \in [n]\setminus (J_+\cup J_-)} \inr{\ol{Y}_j,u} \right)~,
\]
and it suffices to obtain an upper estimate on
\[
\left| \frac{1}{n-2n\theta} \sum_{j \in [n] \backslash (J_+ \cup J_-)} \inr{\ol{Y}_j,u} \right|~.
\]
That is precisely the question addressed in Theorem
\ref{thm:main-integral} for $p=1$ (and with the sample size being $n$
rather than $N$). To apply Theorem \ref{thm:main-integral} and the
subsequent Corollary \ref{thm:main-single-integral}, one needs to
ensure that the random variables $\inr{\ol{Y},u}$ satisfy properties
$(1)$-$(3)$ of Definition \ref{def:ratioconditions}.

Observe that $\left\|\inr{\ol{Y,u}}\right\|_{L_2} = \left\|u\right\|_{L_2}=\sigma(u)$
because the $L_2$ norm endowed by $\ol{Y}$ coincides with the one
endowed by $\ol{X}$.  \eqref{eq:fact1} in Lemma
\ref{lem:normequivalence} shows that property $(3)$ holds
for $\eta=1/4$ if $m \geq m_0(q,\kappa)$ for a constant $m_0(q,\kappa)$.

Also, by \eqref{eq:fact4}, for any $\gamma\in (0,1)$, if $m \geq
c_1/\gamma^2$ for some constant $c_1$, then for any
interval $I \subset \R$,
\[
\PROB\left\{ \ol{Y} \in I \right\} \leq \max\left\{L \frac{|I|}{\sigma(u)}, \gamma\right\}~,
\]
and therefore, with such a choice of $m$, the class
$\F= \left\{\inr{\ol{Y},u}: u \in B_2^d \right\}$ satisfies the
assumptions of Theorem \ref{thm:uniform-properties}.

Now set $\Delta = \gamma/c_1$ (with $c_1$ as in the statement of
Theorem \ref{thm:uniform-properties}) and choose $\theta \ge 7\Delta$.
Let $\rho_1 \ge \ol{\rho}_n(c,\Delta,L)$ where $\ol{\rho}_n(c,\Delta,L)$ is the
critical level of the class $\F$ (with $c$ as in Theorem
\ref{thm:uniform-properties}).

The theorem implies that there is an event ${\cal A}$ with
probability at least $1-2\exp(-c_2\Delta n)$, such that for all
$\|u\|_{L_2} \geq \rho_1$, the random variable $\inr{v,\ol{Y}}$
satisfies properties $(1)$ and $(2)$ of Definition
\ref{def:ratioconditions}.
Therefore, Theorem \ref{thm:main-single-integral} shows that on that event, if
$\|u\|_{L_2} \geq \rho_1$, then
\[
\left| \frac{1}{n-2n\theta} \sum_{j \in [n] \setminus (J_+ \cup J_-)} \inr{\ol{Y}_j,u} \right| \leq c_3 \sqrt{\Delta\log(1/\Delta)} \sigma(u) \leq c_4 \sigma(u)~.
\]
In particular, there is a constant $c(\kappa,q)$ such that, if
\[
m =c(\kappa,q)\frac{N}{\log(1/\delta)}~,
\]
then $\PROB\{{\cal A}\} \geq 1-\delta/2$ and on the event ${\cal A}$, for any $u \in S^{d-1}$ for which $\sigma(u) \geq \rho_1$,
\begin{equation} \label{eq:mean-est-2}
\left|\wh{\nu}_N(u) - \inr{\mu,u} \right| \leq c'(\kappa,q) \sigma(u) \sqrt{\frac{\log(e/\delta)}{N}}~,
\end{equation}
satisfying \eqref{eq:goodeventmarginal}.

It remains to check that the inequality also holds for those $u\in
S^{d-1}$ with $\sigma(u) < \rho_1$.
Let $\rho_2 \geq \rho_1$ to be specified in what follows.
 Clearly, \eqref{eq:mean-est-2} holds when $\sigma(u) \geq
 \rho_2$. When $\sigma(u) < \rho_2$,
one may repeat the argument used in the last part of the proof of
Theorem \ref{thm:iso-variance} (``below the critical level''). It is evident that if
\[
\rho_2 = c(\kappa,q)\sqrt{\frac{\sum_{i \geq c'(\kappa,q)n}  \lambda_i}{n}}~,
\]
for some constants $c(\kappa,q), c'(\kappa,q)$,
then, with probability $1-2\exp(-c_1 n) \ge 1-\delta/2$, %
\[
\sup_{u \in B_2^d \cap \rho_2 D} \left|\inr{\ol{Y}_j,u}\right|^{\sharp}_{\theta
  n} \le c(\kappa,q)\rho_2~.
\]
Therefore, on this event,
\[
\sup_{v \in B_2^d \cap \rho D} \left| \frac{1}{n-2n\theta} \sum_{j \in [n] \setminus (J_+ \cup J_-)} \inr{\ol{Y}_j,u} \right| \leq c^\prime(\kappa,q)\rho_2~,
\]
and
\[
\left|\wh{\nu}_N(u) - \inr{\mu,u} \right| \leq \frac{1}{\sqrt{m}} \cdot c^{\prime \prime} (\kappa,q)\rho_2~.
\]
The announced bound  now follows for all $u$, on an event of
probability at least $1-\delta$.
\endproof

\section{The connection with strong-weak norm inequalities} \label{sec:strong-weak}

Strong-weak norm inequalities are a natural way of quantifying the
tail behaviour of random vectors. Given a random vector $X$ in $\R^d$
and a norm $\| \cdot \|$, we say that $X$ satisfies a strong-weak inequality with constant $C$ if for every $p \geq 1$,
$$
\left(\E \|X-\EXP X\|^p \right)^{1/p} \leq C \left( \E \|X- \EXP X\| + \sup_{z^* \in B^*} \left(\EXP |x^*(X- \EXP X)|^p \right)^{1/p} \right)~,
$$
where $B^*$ is the unit ball of the dual space to $(\R^d, \| \cdot
\|)$. In other words, the way $X$ concentrates around its mean with
respect to the norm $\| \cdot \|$ is governed by the $L_1$ norm of
$\|X- \EXP X\|$ and the largest $L_p$ norm of all the one-dimensional
marginals of the centred random vector $X - \EXP X$.

The fact that the $L_p$ norm of $\|X-\EXP X\|$ can be controlled by
such a combination and no information on the $L_p$ norms of higher
dimensional marginals is a rather powerful feature. The best type of a
strong-weak inequality one can hope for is a \emph{sub-Gaussian} one,
that is, if the tails of one dimensional marginals of $X-\EXP X$ decay
at least as fast as those of a Gaussian:
$$
\sup_{z^* \in B^*} \left(\E |x^*(X- \EXP X)|^p \right)^{1/p} \leq \sqrt{p} \sup_{z^* \in B^*} \sigma(z^*)~,
$$
where $\sigma(z^*)= \left(\E (z^*(X-\EXP X))^2\right)^{1/2}$ is the
variance of the one dimensional marginal defined by $z^*$. In such a
case, an equivalent ``in-probability" version of the strong-weak inequality is that for any $0 < \delta <1/2$,
\begin{equation} \label{eq:in-prob-subgaussian-sw}
\PROB \left( \|X-\EXP X\| \geq C \left( \E \|X-\EXP X\| + \sqrt{\log(1/\delta)} \sup_{z^* \in B^*} \sigma(z^*) \right)  \right) \leq \delta~.
\end{equation}
Clearly, if the one-dimensional marginals of $X-\EXP X$ do not exhibit a subgaussian tail decays, there is no hope that \eqref{eq:in-prob-subgaussian-sw} can be true, even, when $\| \cdot \|$ is the Euclidean norm in $\R^d$, which is our main focus.

When the random vector $Y_N=N^{-1}\sum_{i=1}^N X_i$ satisfies a subgaussian ``in probability" version of the strong-weak inequality, that implies that the empirical mean is an optimal mean estimation procedure. However, almost no random vectors satisfy that strong property. At the same time, \eqref{eq:subgauss} shows that by replacing the empirical mean with $\wh{\mu}_N$, every random vector satisfies a version of a subgaussian strong-weak inequality. Moreover, Theorem \ref{thm:meanest} shows that under a minimal norm equivalence condition, the weak term can be replaced by the optimal directional-dependent term.

\bibliographystyle{plain}

\begin{thebibliography}{10}

\bibitem{Bah20}
Sohail Bahmani.
\newblock Nearly optimal robust mean estimation via empirical characteristic
  function.
\newblock {\em arXiv preprint arXiv:2004.02287}, 2020.

\bibitem{BoLuMa13}
S.~Boucheron, G.~Lugosi, and P.~Massart.
\newblock {\em Concentration inequalities: A Nonasymptotic Theory of
  Independence}.
\newblock Oxford University Press, 2013.

\bibitem{Cat16}
Olivier Catoni.
\newblock Pac-bayesian bounds for the gram matrix and least squares regression
  with a random design.
\newblock {\em arXiv preprint arXiv:1603.05229}, 2016.

\bibitem{ChFlBa19}
Y.~Cherapanamjeri, N.~Flammarion, and P.~Bartlett.
\newblock Fast mean estimation with sub-gaussian rates.
\newblock {\em arXiv preprint arXiv:1902.01998}, 2019.

\bibitem{DaMi20}
Arnak~S Dalalyan and Arshak Minasyan.
\newblock All-in-one robust estimator of the gaussian mean.
\newblock {\em arXiv preprint arXiv:2002.01432}, 2020.

\bibitem{DeGi99}
V.H. de~la Pe\~na and E.~Gin\'e.
\newblock {\em Decoupling: from Dependence to Independence}.
\newblock Springer, New York, 1999.

\bibitem{DeLe19}
J.~Depersin and G.~Lecu{\'e}.
\newblock Robust subgaussian estimation of a mean vector in nearly linear time.
\newblock {\em arXiv preprint arXiv:1906.03058}, 2019.

\bibitem{DiKaPe20}
Ilias Diakonikolas, Daniel~M Kane, and Ankit Pensia.
\newblock Outlier robust mean estimation with subgaussian rates via stability.
\newblock {\em arXiv preprint arXiv:2007.15618}, 2020.

\bibitem{GiZi84}
E.~Gin\'e and J.~Zinn.
\newblock Some limit theorems for empirical processes.
\newblock {\em Annals of Probability}, 12:929--989, 1984.

\bibitem{MR2243881}
Evarist Gin\'{e} and Vladimir Koltchinskii.
\newblock Concentration inequalities and asymptotic results for ratio type
  empirical processes.
\newblock {\em Ann. Probab.}, 34(3):1143--1216, 2006.

\bibitem{Giu18}
Ilaria Giulini.
\newblock Robust dimension-free {G}ram operator estimates.
\newblock {\em Bernoulli}, 24(4B):3864--3923, 2018.

\bibitem{Hop18}
S.B. Hopkins.
\newblock Sub-gaussian mean estimation in polynomial time.
\newblock {\em Annals of Statistics}, 2019, to appear.

\bibitem{KoLo17}
Vladimir Koltchinskii and Karim Lounici.
\newblock Concentration inequalities and moment bounds for sample covariance
  operators.
\newblock {\em Bernoulli}, 23(1):110--133, February 2017.

\bibitem{MR2449135}
R.~Lata{\l}a and J.~O. Wojtaszczyk.
\newblock On the infimum convolution inequality.
\newblock {\em Studia Math.}, 189(2):147--187, 2008.

\bibitem{LeTa91}
M.~Ledoux and M.~Talagrand.
\newblock {\em Probability in Banach Space}.
\newblock Springer-Verlag, New York, 1991.

\bibitem{LeLuVeZh20}
Zhixian Lei, Kyle Luh, Prayaag Venkat, and Fred Zhang.
\newblock A fast spectral algorithm for mean estimation with sub-gaussian
  rates.
\newblock In {\em Conference on Learning Theory}, pages 2598--2612, 2020.

\bibitem{Lou14}
Karim Lounici.
\newblock High-dimensional covariance matrix estimation with missing
  observations.
\newblock {\em Bernoulli}, 20(3):1029--1058, 2014.

\bibitem{LuMe19}
G.~Lugosi and S.~Mendelson.
\newblock Mean estimation and regression under heavy-tailed distributions---a
  survey.
\newblock {\em Foundations of Computational Mathematics}, 2019.

\bibitem{LuMe16a}
G.~Lugosi and S.~Mendelson.
\newblock Sub-{G}aussian estimators of the mean of a random vector.
\newblock {\em Annals of Statistics}, 47:783--794, 2019.

\bibitem{LuMe20}
G.~Lugosi and S.~Mendelson.
\newblock Robust multivariate mean estimation: the optimality of trimmed mean.
\newblock {\em Annals of Statistics}, to appear, 2020.

\bibitem{MeZh18}
S.~Mendelson and N.~Zhivotovskiy.
\newblock Robust covariance estimation under ${L}_4-{L}_2$ norm equivalence.
\newblock {\em Annals of Statistics}, 48(3):1648–1664, 2020.

\bibitem{Men18a}
Shahar Mendelson.
\newblock Approximating the covariance ellipsoid.
\newblock {\em arXiv preprint arXiv:1804.05402}, 2018.

\bibitem{Men20}
Shahar Mendelson.
\newblock Approximating $ l\_p $ unit balls via random sampling.
\newblock {\em arXiv preprint arXiv:2008.08380}, 2020.

\bibitem{Men20a}
Shahar Mendelson.
\newblock Learning bounded subsets of $ l\_p $.
\newblock {\em arXiv preprint arXiv:2002.01182}, 2020.

\bibitem{Min18}
Stanislav Minsker.
\newblock Sub-{G}aussian estimators of the mean of a random matrix with
  heavy-tailed entries.
\newblock {\em The Annals of Statistics}, 46(6A):2871--2903, 2018.

\bibitem{MiMa19}
Stanislav Minsker and Timoth{\'e}e Mathieu.
\newblock Excess risk bounds in robust empirical risk minimization.
\newblock {\em arXiv preprint arXiv:1910.07485}, 2019.

\bibitem{MiNd20}
Stanislav Minsker and Mohamed Ndaoud.
\newblock Robust and efficient mean estimation: approach based on the
  properties of self-normalized sums.
\newblock {\em arXiv preprint arXiv:2006.01986}, 2020.

\bibitem{MiWe20}
Stanislav Minsker and Xiaohan Wei.
\newblock Robust modifications of u-statistics and applications to covariance
  estimation problems.
\newblock {\em Bernoulli}, 26(1):694--727, 2020.

\bibitem{MoODOl10}
E.~Mossel, R.~O'Donnell, and K.~Oleszkiewicz.
\newblock Noise stability of functions with low influences: invariance and
  optimality.
\newblock {\em Annals of Mathematics}, 171:295--341, 2010.

\bibitem{Tal96c}
M.~Talagrand.
\newblock New concentration inequalities in product spaces.
\newblock {\em Inventiones Mathematicae}, 126:505--563, 1996.

\bibitem{vaWe96}
A.W. {van der Vaart} and J.A. Wellner.
\newblock {\em Weak convergence and empirical processes}.
\newblock Springer, 1996.

\end{thebibliography}

\end{document}